\documentclass[11pt,a4paper]{article}
\usepackage{indentfirst}
\usepackage{fancyhdr}
\usepackage{color}
\usepackage{mathrsfs}
\usepackage{amsfonts,amssymb,amsmath,amsthm}
\usepackage{authblk}
\usepackage[unicode,pdftex]{hyperref}
\hypersetup{
	colorlinks = true,
	citecolor = green,
	anchorcolor = blue,
	linkcolor = blue}

\newtheorem{thm}{Theorem}[section]
\newtheorem{lemma}[thm]{Lemma}

\newtheorem{prop}[thm]{Proposition}

\newtheoremstyle{rem}{10pt}{10pt}{\rmfamily}{}{\bfseries}{.}{.5em}{}
\theoremstyle{rem}
\newtheorem{rem}[thm]{Remark}

\numberwithin{equation}{section} 

\title{Stochastic Schr\"odinger-Korteweg
	de Vries systems driven by multiplicative noises}
\author[1]{Jie Chen}
\author[2]{Fan Gu\thanks{Corresponding Author: gufan@amss.ac.cn}}
\author[3]{Boling Guo}
\affil[1]{\scriptsize \textit{School of Sciences, Jimei University, Xiamen 361021, P.R. China}}
\affil[2]{\scriptsize \textit{School of Statistics and Mathematics, Central University of Finance and Economics,  Beijing 102206, P.R. China}}
\affil[3]{\scriptsize \textit{Institute of Applied Physics and Computational Mathematics, Beijing 100088, P.R. China}}
\date{}

\begin{document}
	\maketitle

\begin{abstract}
In this paper, we consider the well-posedness of stochastic S-KdV driven by multiplicative noises in $H_x^1\times H_x^1$. To get the local well-posedness, we first develop the bilinear and trilinear Bourgain norm estimates of the nonlinear terms with $b\in\left(0,1/2\right)$. Then, to overcome regularity problems, we introduce a series of approximation equations with localized nonlinear terms, which are also cutted-off in both the physical and the frequency space. By limitations, these approximation equations will help us get a priori estimate in the Bourgain space and finish the proof of the global well-posedness of the initial system. 
\end{abstract}

\section{Introduction}\label{sec:introduction}

In this paper, we study the local and global well-posedness of stochastic Schr\"{o}dinger-Korteweg
de Vries systems driven by  multiplicative noises in $H_x^1\times H_x^1$.
We consider the following \eqref{msskdv} in $t\in[0,T_0]$:
\begin{equation}\label{msskdv}
	\tag{S-S-KdV}
	\left\{
	\begin{aligned}
		&d u=i\partial_{xx}udt -i (\gamma_1 uw+\beta|u|^2u)dt+F(u)^{\alpha} \Phi dW_t^{(1)},\\
		&d w=-\partial_{xxx}wdt +(\gamma_2\partial_x(|u|^2)-w\partial_x w)dt+w^{\alpha}\Psi dW_t^{(2)};\\
		&(u,w)|_{t = 0} = (u_0,w_0)\in H^1(\mathbb{R})\times H^1(\mathbb{R}),
	\end{aligned}
	\right.
\end{equation}
where $\alpha, \beta, \gamma_1, \gamma_2$ are real-valued constants, $\alpha\in \mathbb{N^+}$ and $\gamma_1\cdot\gamma_2>0$. $u,v$ are  complex-valued and  real-valued, respectively. $F(u)$ can be $u,\ \bar{u},\ \text{Im}u  $ or $\text{Re} u$.  

\eqref{msskdv} is defined in a filtrated probability space $\left(\Omega, \mathscr{F},\mathbb{P},\left(\mathscr{F}_t\right)_{t\in\left[0,T_0\right]} \right)$. $\frac{dW_t^{(1)}}{dt}, \frac{dW_t^{(2)}}{dt}$ are two space-time independent white noises on $L^2\left( \mathbb{R}\right)$ adapted to $\left\{\mathscr{F}_t\right\}_{t\in\left[0,T_0\right]}$. $W^{(r)}_t, r=1,2$ can be represented as $\sum_{k=0}^{+\infty}\beta_k^{(r)}(t)e_k^{(r)}$, where $\{\beta_k^{(r)}\}$ are two sequences of mutually independent real standard Brownian motions and $\{e_k^{(r)}\}$ are two orthonormal basis of real-valued $L^2(\mathbb{R})$. $\Phi,\Psi$ are two homogeneous convolution operators on $L^2(\mathbb{R})$ defined by
$$
\Phi f(x)=\int_{\mathbb{R}} k_1(x-y)f(y)dy,\  \Psi f(x) = \int_{\mathbb{R}} k_2(x-y)f(y)dy,\  \forall f\in L^2(\mathbb{R}).
$$

The deterministic S-KdV is devoted to describe the interactions between short waves $u(x,t)$ and long waves $v(x,t)$ in  fluid mechanics and plasma physics. 
The case $\beta=0$ appears in the study of resonant
interaction between short and long capillary-gravity waves on water of a uniform finite depth, in plasma physics and in a diatomic lattice system. For more background details, one can see \cite{77dynamics}, \cite{77theory} and \cite{corcho2007well}. 

The well-posedness of the deterministic coupled system has been widely researched. In \cite{Guo}, \cite{H1H1}, \cite{corcho2007well} and \cite{guo2010well}, they have proved the well-posedness in $C([0,T];H^s(\mathbb{R})\times H^s(\mathbb{R})), \ s \geq 3$; $C([0,T];H^s(\mathbb{R})\times H^s(\mathbb{R}))$, $s\in\mathbb{N^+}, \beta=0$; $C([0,T];L^2(\mathbb{R})\times H^{(-3/4)+}(\mathbb{R}))$, $s\geq\frac{1}{2}$ and $C([0,T];L^2(\mathbb{R})\times H^{-3/4}(\mathbb{R}))$, respectively.

For \eqref{msskdv}, the multiplicative noises can be interpreted as external random potentials or stochastic dissipative terms, see \cite{de1999stochastic} \cite{schrodingerH1} \cite{14zhang}  \cite{kdvmuti} and \cite{de09} for examples.
Since there is a nonlinear term with derivative in S-KdV, the well-posedness of the stochastic S-KdV is more like the well-posedness of stochastic KdV, such as \cite{KdVH1} \cite{de1999white} \cite{kdvmuti}. 
Roughly speaking, the proof of the well-posedness of stochastic S-KdV can be divided into to two steeps. 

The first steep is the local well-posedness, including the choice of work-space, estimates of nonlinear terms and estimates of stochastic integral terms. 
In \cite{chenguguo}, the additive noises case, we use a workspace with the maximum functional estimate. In  multiplicative cases, if we still use that workspace, we need to estimate the  $\|\cdot\|_{L^2_\omega L_x^2L_t^\infty}$-norm of the stochastic integral terms. For feasibility, we may estimate $\|\cdot\|_{L^2_\omega L_x^2W_t^{\alpha,p}},\ \alpha p>1,\ p>2$. In general, to estimate $\|\cdot\|_{L^2_\omega L_x^2W_t^{\alpha,p}} $, we need to estimate $\|\cdot\|_{ L_x^2W_t^{\alpha,p}L^p_\omega}$. But this norm is hard to be controlled by $\|\cdot\|_{L^2_\omega L_x^2W_t^{\alpha,p}} $ or $ \|\cdot\|_{L^2_\omega L_t^\infty H_x^1}$. Thus, we use the Bourgain norms with the regularity of time lower than $1/2$ to overcome this problem.  To deal with coupling terms by Bourgain method, we develop some mixing bilinear and multilinear estimates of Schr\"odinger and KdV equations. Furthermore, in the multiplicative case, because there are $u,v$ in the noise terms, we must use the workspace with expectation. This fact requires us to localize nonlinear terms in time. Then, a fixed point argument can be done.

The second step is to prove the global well-posedness. However, when we use conservation laws, there are problems in the regularity. To overcome these problems, 
we propose a series of approximation equation with nonlinear terms cut-off in both the physical space and the frequency space.  These approximation equations are effective because they are global well-posedness in a high regularity space until any constant time. At the same time, we can also get an $H_x^1\times H_x^1$ priori estimate of the approximation equations until any constant time, since they have enough conservation laws(Proposition \ref{prop:uniform_esti}).  Furthermore, through a path-wise view, this estimate will transfer to a priori estimate of Bourgain norms. By limitation, these estimates will help us finish the proof of the global well-posedness.

This paper is organized as the following manner: First, we introduce some definitions and notations. Our main results are also proposed in this section. Second, we prove some  bilinear estimates and multilinear estimates in the Bourgain space with $b<1/2$. Some estimates of stochastic integrals are also given in this section too. Next, we use above estimates to prove the local well-posedness by a fixed point argument. Finally, to prove the global well-posedness, we introduce a sequences of approximation equations, which helps us finish the proof of the main results, by considering their priori estimates and convergences.

\section{Preliminaries and the Main Result}\label{sec:preliminary}
In this section, we introduce some notations, definitions and basic facts.

For any $ \varphi(x,t) \in \mathscr{S}'(\mathbb{R}\times\mathbb{R})$, we use $\hat{\varphi}(\xi,\tau)$ or $ \mathcal{F}_{x,t}[\varphi](\xi,\tau)$ to represent the space-time Fourier transform of $\varphi$. We use $ \mathcal{F}_x[\varphi](\xi,t)$ to denote the space Fourier transform of $\varphi$. The inverse Fourier transforms are  denoted by $\check{\varphi}(x,t)$, $ \mathcal{F}^{-1}_{\xi,\tau}[\varphi](x,t)$ and $F^{-1}_{\xi}[\varphi](x) $ similarly.

For any $f\in L^2(\mathbb{R})$, $n\in\mathbb{N}^+$, we denote $J^n f(x) =F^{-1}_{\xi}\left[(1+|\xi|)^n\hat{f}(\xi)\right](x)$.

$H^s(\mathbb{R}),\ \dot{H}^{s}(\mathbb{R}),\ s\in\mathbb{R} $ are the usual Sobolev space and the homogeneous Sobolev sapce, respectively. We also use $\mathcal{H}^1(\mathbb{R}) $ to represent $H^1(\mathbb{R})\times H^1(\mathbb{R})$.

For any $f\in L^2(\mathbb{R})$, $m>0$, we denote the low frequency cut-off operator and the high frequency cut-off operator  by 
$$
P_mf(x)=\mathcal{F}^{-1}_x[ \chi_{[-m,m]}(\xi)\mathcal{F}_x[f]](x),\ P_{\geq m } f(x)=\mathcal{F}^{-1}_x[ \chi_{[-m,m]^c}(\xi)\mathcal{F}_x[f]](x),
$$
respectively.

For any $s>0,\ f\in H^s(\mathbb{R}) $, by Lemma 2.6 in \cite{kdvmuti}, we have conditions 
$$ \sum_{k=0}^{\infty} \| f\Phi e_k\|_{H_x^s}^2 <\infty,\ \sum_{k=0}^{\infty} \| f\Psi e_k\|_{H_x^s}^2 <\infty $$
can be replaced by  condition $\|k_1\|_{H_x^s}<\infty$ and $\|k_2\|_{H_x^s}<\infty$, respectively. 

Moreover, it can be proved that conditions
$$ \sum_{k=0}^{\infty} \| fP_m\Phi e_k\|_{H_x^s}^2 <\infty,\ \sum_{k=0}^{\infty} \| P_m(f\Psi e_k)\|_{H_x^s}^2 <\infty $$
can also be replaced by conditions $\|k_1\|_{H_x^s}<\infty$ and $\|k_2\|_{H_x^s}<\infty$.

For $f,g\in \mathscr{S}'(\mathbb{R}^2)$, we denote the Bourgain norms by
$$
\|f\|_{X_{b,s}}=\left(\int_{\mathbb{R}^2}(1+|\xi|)^{2s} \left(1+|\tau+\xi^2|\right)^{2b} |\hat{f} (\xi,\tau)|^2d\xi d\tau\right)^{1/2},
$$
$$
\|g\|_{Y_{b,s}}=\left(\int_{\mathbb{R}^2}(1+|\xi|)^{2s} \left(1+|\tau-\xi^3|\right)^{2b}|\hat{g} (\xi,\tau)|^2d\xi d\tau\right)^{1/2},
$$
$$
\|g\|_{Y_{b,s,-\frac{3}{8}}}=\left(\int_{\mathbb{R}^2}|\xi|^{-\frac{3}{4}}(1+|\xi|)^{2s} (1+|\tau-\xi^3|)^{2b}|\hat{g} (\xi,\tau)|^2d\xi d\tau\right)^{1/2}.
$$
The restricted norms are denoted  by
$$
\|f\|_{X^T_{b,s}}=\inf\{\|\tilde{f}\|_{X_{b,s}}\}, \ \|g\|_{Y^T_{b,s}}=\inf\{\|\tilde{g}\|_{Y_{b,s}}\}, \
\|g\|_{Y^T_{b,s,-3/8}}=\inf\{\|\tilde{g}\|_{Y_{b,s,-3/8}}\},
$$
for any $\tilde{f}=f, \tilde{g}=g $ in $[0,T]$.
For the sake of simplicity, we also denote 
$$
\tilde{Y}_{b,s}= Y_{b,s}\cap Y_{b,s,-3/8},\   \tilde{Y}^T_{b,s}= Y^T_{b,s}\cap Y^T_{b,s,-3/8}
$$
and
$$
\left<C\right>=1+|C|.
$$

From the Lemma 2.1 in \cite{kdvmuti}, we know that there exist constants $C_1$ and $C_2$ depending only on $b$ such that
$$
C_1\|f\|_{X_{b,s}^T}\leq \|\chi_{[0,T]}f\|_{X_{b,s}} \leq C_2\|f\|_{X_{b,s}^T},\ C_1\|g\|_{Y_{b,s}^T}\leq \|\chi_{[0,T]}g\|_{Y_{b,s}} \leq C_2\|g\|_{Y_{b,s}^T},
$$ 
for any $b\in[0,1/2)$, $T,s\geq0$, $f\in X^T_{b,s}$ and $g\in Y^T_{b,s} $.

The mild solution of \eqref{msskdv} is
\begin{equation}\label{mild}
	\left\{
	\begin{aligned}
		&u(t)=S(t)u_0-i\int_{0}^{t} S(t-r)\left(\gamma_1uw+\beta|u|^2u\right)dr+\int_{0}^{t} S(t-r)F(u)^{\alpha}\Phi dW_r^{(1)}\\
		&w(t)=U(t)w_0+\int_{0}^{t} U(t-r)\partial_x\left(\gamma_2|u|^2-\frac{1}{2}w^2\right)dr+\int_{0}^{t}U(t-r) w^{\alpha}\Psi dW_r^{(2)},
	\end{aligned}
	\right.
\end{equation}
where $S(t)$ and $U(t)$ are unity semi-groups of the deterministic linear Schr\"{o}dinger equation and linear KdV equation, respectively. 

However,  since the nonlinear part of $w$ actually has a better regularity in $C\left([0,T_0];\dot{H}^{-3/8}_x\right)$, we also consider the following equation:
\begin{equation}\label{mild_without_linear}
	\left\{
	\begin{aligned}
		&u(t)=S(t)u_0-i\int_{0}^{t} S(t-r)\left(\gamma_1u\left(v+U(r)w_0\right)+\beta|u|^2u\right)dr\\
		&\qquad\qquad+\int_{0}^{t} S(t-r)F(u)^{\alpha}\Phi dW_r^{(1)},\\
		&v(t)=\int_{0}^{t} U(t-r)\partial_x\left(\gamma_2|u|^2-\frac{1}{2}\left(v+U(r)w_0\right)^2\right)dr\\
		&\qquad\qquad+\int_{0}^{t}U(t-r) \left(v+U(r)w_0\right)^{\alpha}\Psi dW_r^{(2)};\\
		&(u,v)|_{t = 0} = (u_0,0).
	\end{aligned}
	\right.
\end{equation}

Let cut-off function $\theta\in C_c^\infty(\mathbb{R})$ satisfy $\text{supp}\ \theta\subset[-2,2]$, $\theta(t)=1,\forall t\in[-1,1]$ and $|\theta(t)|\leq 1,\forall t\in \mathbb{R}$. We also denote $\theta_R(t)=\theta(\frac{t}{R})$ and $ \tilde{u}_R(x,t)=\theta_R(\|u_R\|_{X^t_{b,1}})u_R(x,t) , \ \tilde{v}_R(x,t)=\theta_R(\|v_R\|_{\tilde{Y}^t_{b,1}})v_R(x,t)$, for proper $b\in(0,1/2)$, which will be figured in following sections.

To prove the local well-posedness of $\eqref{msskdv}$, we need the following localized equation:
\begin{equation}\label{truncated_eq}
	\left\{
\begin{aligned}
	&u_R(t)=S(t)u_0-i\int_{0}^{t} S(t-r)\left(\gamma_1\tilde{u}_R(\tilde{v}_R+U(r)w_0)+\beta|\tilde{u}_R|^2\tilde{u}_R\right)dr\\
	&\qquad \quad +\int_{0}^{t} S(t-r)F(\tilde{u}_R)^{\alpha}\Phi dW_r^{(1)},\\
	&v_R(t)=\int_{0}^{t} U(t-r)\partial_x\left(\gamma_2|\tilde{u}_R|^2-\frac{1}{2}\left(\tilde{v}_R+U(r)w_0\right)^2\right)dr\\
	&\qquad \quad+\int_{0}^{t}U(t-r)(v_R+U(r)w_0)^{\alpha}\Psi dW_r^{(2)}.
\end{aligned}
\right.
\end{equation}
For any $R>0$, we define stopping times
$$
\sigma_R^{(1)}=\inf\left\{t\geq 0: \|u_R\|_{X^t_{b,1}}\geq R\right\},\ \sigma_R^{(2)}=\inf\left\{t\geq 0: \|v_R\|_{\tilde{Y}^t_{b,1}}\geq R \right\}.
$$
and $w_R(t)=U(t)w_0+v_R(t)$. 

Now, we can propose our main theorems of this paper. The first theorem is  the local well-posedness and the the second theorem is  the global well-posedness.

\begin{thm}\label{thm:local_well-posedness}
	Suppose that $(u_0,w_0)\in H_x^1\times  H_x^1$, $k_1\in H_x^1$, $k_2\in H_x^1\cap L_x^1$. Let us define $b_\alpha=\left(1/2-1/4(\alpha-1)\right)\vee3/8$. Then for any $R,T_0>0$, $\alpha\in \mathbb{N}^+$, $b\in(b_\alpha,1/2)$, $l\in\mathbb{N^+}$, there exists a unique strong  solution $(u_R,v_R)\in L^{2l}\left(\Omega;X^{T_0}_{b,1}\times \tilde{Y}^{T_0}_{b,1}\right)$ of \eqref{truncated_eq}. Moreover, we have $(u_R,v_R)\in L^{2l}\left(\Omega;C\left([0,T_0];H_x^1\times H_x^1\right)\right)$.
\end{thm}

\begin{thm}\label{thm:main}
	Suppose that $(u_0, w_0)\in H_x^1\times H_x^1$, $k_1 \in  H^1_x$, $ k_2 \in L^1_x\cap H^1_x$. Then for any $l\in\mathbb{N^+}$, $T_0>0$, \eqref{msskdv} has a unique strong solution in $L^{2l}\left(\Omega;C\left([0,T_0];H_x^1\times H_x^1\right)\right)$ with $\alpha=1$.
\end{thm}

We should notice that although our initial values are deterministic, by the standard method in \cite{de1999stochastic} or \cite{schrodingerH1}, our result can be generalized to the case  $(u_0,w_0)\in H_x^1\times H_x^1$ almost surely. Correspondingly, the result will turn to $(u,w)\in C\left([0,T_0];H_x^1\times H_x^1\right)$ almost surely.

Without loss of generality, we let $\beta=\gamma_1=\gamma_2=1$ in the following paper.

\section{Fundamental Inequalities}\label{sec:ineq}
To deal with the nonlinear terms, we need some bilinear and trilinear estimates in the Bourgain space. 

Since $\text{Im}u= -i/2\cdot(u-\bar{u})$ and $\text{Re} u=1/2(u+\bar{u})$, in this and the next section, we let $F(u)$ be $u$ or $\bar{u}$.

We first concern the coupling term in the Schr\"{o}dinger-type equation. The following basic inequality will be used.
\begin{lemma}\label{lem:basic_inequality_homogeneous}
	For any $a,b\in(1/4,1/2)$, we have
	$$
	\int_{\mathbb{R}} \frac{1}{\left<x-\alpha\right>^{2a}\left<x-\beta\right>^{2b}}  dx \leq \frac{C}{\left< \alpha-\beta\right>^{2a+2b-1}} .
	$$
\end{lemma}
\begin{proof}
	This can be proved by the fact
	$$
	\begin{aligned}
		&\int_{\mathbb{R}} \frac{1}{\left<x-\alpha\right>^{2a}\left<x-\beta\right>^{2b}}  dx \\
		=&\int_{\mathbb{R}} \frac{1}{\left<x\right>^{2a}\left<x+\alpha-\beta\right>^{2b}}  dx\\
		\leq&	2\int_{|\alpha-\beta|/2}^{+\infty}\frac{dx}{(1+x)^{2a+2b}}+2\int_{0}^{|\alpha-\beta|/2}\frac{dx}{(1+x)^{2a}(1+|\alpha-\beta|/2)^{2b}}\\
		&+2\int_{0}^{|\alpha-\beta|/2}\frac{dx}{(1+x)^{2b}(1+|\alpha-\beta|/2)^{2a}}.
	\end{aligned}
	$$
\end{proof}

\begin{lemma}\label{lem:KS_to_S}
	Suppose $a,b\in(1/4,1/2)$ and $ a+2b>1$. Then for any $g\in X_{b,1},\  h\in Y_{b,1}$, we have
	\begin{equation}\label{KS_to_S}
		\|gh\|_{X_{-a,1}}\leq C \| g\|_{X_{b,1}}\| h\|_{Y_{b,1}}.
	\end{equation}
\end{lemma}
\begin{proof}
	By the duality, we only need to prove that for $f,g,h\in\mathscr{S}(\mathbb{R}^2)$ 
	$$
	|(f,gh)|\leq C \|f\|_{X_{a,-1}}\|g\|_{X_{b,1}}\|h\|_{Y_{b,1}}.
	$$
	By the Plancherel Theorem, we will estimate
	\begin{equation*}
		\begin{aligned}
			&\int_{\mathbb{R}^4} \frac{\left< \tau+\xi^2\right>^a}{\left<\xi \right>}\hat{f}(\xi,\tau)\left<\tau-\tau_1+(\xi-\xi_1)^2\right>^b\left<\xi-\xi_1\right>\bar{\hat{g}}(\tau-\tau_1,\xi-\xi_1)\\
			&\frac{\left<\xi_1\right>\left<\tau_1-\xi_1^3\right>^b\bar{\hat{h}}(\tau_1,\xi_1)}{\left< \tau+\xi^2\right>^a\left<\tau-\tau_1+(\xi-\xi_1)^2\right>^b\left<\tau_1-\xi_1^3\right>^b}d\xi d\xi_1 d\tau d\tau_1.
		\end{aligned}
	\end{equation*}
	Note that $\left<\xi\right>\leq\left<\xi- \xi_1\right>\left<\xi_1\right>$, it is sufficient to estimate
	\begin{equation}\label{esti_KS_to_S}
		\begin{aligned}
			&\int_{\mathbb{R}^4} \left< \tau+\xi^2\right>^a\hat{f}(\xi,\tau)\left<\tau-\tau_1+(\xi-\xi_1)^2\right>^b\bar{\hat{g}}(\tau-\tau_1,\xi-\xi_1)\\
			&
			\frac{\left<\tau_1-\xi_1^3\right>^b\bar{\hat{h}}(\tau_1,\xi_1)}{\left< \tau+\xi^2\right>^a\left<\tau-\tau_1+(\xi-\xi_1)^2\right>^b\left<\tau_1-\xi_1^3\right>^b}d\xi d\xi_1 d\tau d\tau_1.
		\end{aligned}
	\end{equation}
	We set 
	$$
	E(\xi,\tau,\xi_1,\tau_1)=\max{\left\{\left< \tau+\xi^2\right>,\ \left<\tau_1-\xi_1^3\right>,\ \left<\tau-\tau_1+(\xi-\xi_1)^2\right>\right\}}
	$$
	and divide $\mathbb{R}^4$ into 
	$$
	\begin{aligned}
		&{\bf Reigon~I.}\ \   \left\{\mathbb{R}^4: \left< \tau+\xi^2\right>\geq 	E(\xi,\tau_1,\xi_1,\tau_1)\right\}; \\
		&{\bf Reigon~II.}\   \left\{\mathbb{R}^4: \left<\tau_1-\xi_1^3\right>\geq 	E(\xi,\tau_1,\xi_1,\tau_1)\right\}; \\
		&{\bf Reigon~III.} \left\{\mathbb{R}^4: \left<\tau-\tau_1+(\xi-\xi_1)^2\right>\geq 	E(\xi,\tau_1,\xi_1,\tau_1)\right\}.
	\end{aligned}
	$$
	{\bf Reigon~I.:} By Cauchy-Schwarz inequality, we only need to prove  
	\begin{equation}\label{KS_to_S_I}
		\begin{aligned}
			&\sup_{\mathbb{R}^2_{\xi,\tau}}\frac{1}{\left< \tau+\xi^2\right>^a}\left(\int_{\left< \tau+\xi^2\right>\geq 	E(\xi,\tau,\xi_1,\tau_1)}\frac{d\xi_1d\tau_1}{\left<\tau-\tau_1+(\xi-\xi_1)^2\right>^{2b}\left<\tau_1-\xi_1^3\right>^{2b}}\right)^{1/2}\\
			=:&\sup_{\mathbb{R}^2_{\xi,\tau}}N_1(\xi,\tau)
			<\infty.
		\end{aligned}
	\end{equation}
	Calculating the integral of $\tau_1$, we have
	\begin{equation}\label{KS_to_S_I_1}
		\begin{aligned}
			&\sup_{\mathbb{R}^2_{\xi,\tau}}N_1(\xi,\tau)\\
			\leq&\sup_{\mathbb{R}^2_{\xi,\tau}}\frac{1}{\left< \tau+\xi^2\right>^a}\left( \int_{\left< \tau+\xi^2\right>\geq 	E(\xi,\tau_1,\xi_1,\tau_1)} \frac{d\xi_1}{\left<\tau+\xi^2+\xi_1^2-2\xi\xi_1-\xi_1^3 \right>^{4b-1}}  \right)^{1/2}.
		\end{aligned}
	\end{equation}
	
	If $|\xi_1|\leq 100$, then we have 
	\begin{equation}\label{KS_to_S_I_result_1}
		\begin{aligned}
			\sup_{\mathbb{R}^2_{\xi,\tau}}N_1(\xi,\tau)\leq C.
		\end{aligned}
	\end{equation}
	
	If $|\xi_1|> 100$ and $|2\xi_1-2\xi-3\xi_1^2|>1$,  then we set $\mu=\tau+\xi^2+\xi_1^2-2\xi\xi_1-\xi_1^3$. Thus, it implies that
	\begin{equation}\label{KS_to_S_I_result_2}
		\begin{aligned}
			\sup_{\mathbb{R}^2_{\xi,\tau}}N_1(\xi,\tau)\leq& \sup_{\mathbb{R}^2_{\xi,\tau}}\frac{C}{\left< \tau+\xi^2\right>^a}\left(\int_{|\mu|<2|\tau+\xi^2|} \frac{d\mu}{\left<\mu\right>^{4b-1}}\right)^{1/2}\\
			\leq&C \left< \tau+\xi^2\right>^{1-2b-a}<\infty.
		\end{aligned}
	\end{equation}
	
	If $|\xi_1|> 100$ and $|2\xi_1-2\xi-3\xi_1^2|\leq1$, then the length of the integral domain of $\xi_1$ is less then $C/100$, which means in this condition  
	\begin{equation}\label{KS_to_S_I_result_3}
		\sup_{\mathbb{R}^2_{\xi,\tau}}\frac{C}{\left< \tau+\xi^2\right>^a}<\infty.
	\end{equation}
	{\bf Reigon~II.:} In this case, if $|\xi_1|\geq 1$, it is sufficient to prove
	\begin{equation}\label{KS_to_S_II_result_1}
		\begin{aligned}
			&\sup_{\mathbb{R}^2_{\xi_1,\tau_1}}\frac{1}{\left< \tau_1-\xi_1^3\right>^b}\left(\int_{\left< \tau_1-\xi_1^3\right>\geq	E(\xi,\tau,\xi_1,\tau_1)}\frac{d\xi d\tau }{\left<\tau-\tau_1+(\xi-\xi_1)^2\right>^{2b}\left<\tau+\xi^2\right>^{2a}}\right)^{1/2}\\
			\leq&\sup_{\mathbb{R}^2_{\xi_1,\tau_1}}\frac{1}{\left< \tau_1-\xi_1^3\right>^b}\left(\int_{|\tau_1-\xi_1^2+2\xi\xi_1|\leq 2|\tau_1-\xi_1^3|}\frac{d\xi }{\left<\tau_1-\xi_1^2+2\xi\xi_1\right>^{2a+2b-1}}\right)^{1/2}\\
			\leq&C\sup_{\mathbb{R}^2_{\xi_1,\tau_1}}\left< \tau_1-\xi_1^3\right>^{1-a-2b}<\infty.
		\end{aligned}
	\end{equation}
	
	If $|\xi_1|< 1$ and $\left< \tau_1-\xi_1^3\right>\sim \left< \tau+\xi^2\right>+\left<\tau-\tau_1+(\xi-\xi_1)^2 \right>$, then this case can be treated like Reigon.I or Reigon.III.

	If $|\xi_1|< 1$ and $\left< \tau_1-\xi_1^3\right>\gg \left< \tau+\xi^2\right>+\left<\tau-\tau_1+(\xi-\xi_1)^2 \right>$, then we have 
	$$
	|\tau_1-\xi_1^3|\sim|\tau_1-\xi_1^3-\tau-\xi^2+\tau-\tau_1+(\xi-\xi_1)^2|=|\xi_1||\xi_1^2+\xi-\xi_1|\gg1
	$$
	and
	$$
	|\xi_1^2+\xi-\xi_1|\gg1,   \ |\xi|\gg1.
	$$
	(If $|\tau_1-\xi_1^3|\sim 1$, the proof is obvious.)
	
	Thus, by Cauchy-Schwarz inequality it is sufficient to prove 
	\begin{equation}\label{KS_to_S_II_result_2}
		\begin{aligned}
		&\sup_{\mathbb{R}^2_{\xi,\tau}}\frac{1}{\left<\tau+\xi^2\right>^a}\left(\int_{|\xi_1|\leq1,\ |\xi_1||\xi_1^2+\xi-\xi_1|\gg1 }\frac{d\xi_1}{\left<\tau+\xi^2-\xi_1^3+\xi_1^2-2\xi\xi_1\right>^{4b-1}}\right)^{1/2}\\
		<&\infty,
		\end{aligned}
	\end{equation}
which is clear.

	{\bf Reigon~III.} In this case, we first let $ \sigma=\tau-\tau_1$ and $ \eta=\xi-\xi_1$.
	Then we will estimate
	\begin{equation}\label{KS_to_S_III}
		\begin{aligned}
			&\sup_{\mathbb{R}^2_{\eta,\sigma}}\frac{1}{\left< \sigma+\eta^2\right>^b}\left(\int_{\left<\sigma+\eta^2\right>\geq 	E(\xi,\tau,\xi_1,\tau_1)}\frac{d\xi_1d\tau_1}{\left<\tau_1-\xi_1^3\right>^{2b}\left<\sigma+\tau_1+(\eta+\xi_1)^2\right>^{2a}}\right)^{1/2}\\
			\leq& \sup_{\mathbb{R}^2_{\eta,\sigma}}\frac{1}{\left< \sigma+\eta^2\right>^b}\left(\int_{\left<\sigma+\eta^2\right>\geq 	E(\xi,\tau,\xi_1,\tau_1)}\frac{d\xi_1}{\left<\xi_1^3+\sigma+(\eta+\xi_1)^2\right>^{2a+2b-1}}\right)^{1/2}\\
			=:&\sup_{\mathbb{R}^2_{\eta,\sigma}}N_3(\eta,\sigma).
		\end{aligned}
	\end{equation}
	Therefore, we can prove $\sup_{\mathbb{R}^2_{\eta,\sigma}}N_3(\eta,\sigma)<\infty$ like Reigon I. 
\end{proof}

The following trilinear estimate is about the cubic term in the  Schr\"{o}dinger-type equation.
\begin{lemma}\label{lem:trilinear_estimate}
	Suppose that $ a,b\in(\frac{3}{8},\frac{1}{2})$, we have 
	\begin{equation}\label{trilinear_estimate}
		\||u|^2u\|_{X_{-a,1}}\leq C\|u\|^3_{X_{b,1}}.
	\end{equation}
\end{lemma}
\begin{proof}
	By duality, we need to estimate for any $f\in X_{a,-1}$ 
	$$
	\begin{aligned}
		&\left(f,|u|^2u\right)\\
		=&\left(\hat{f},\hat{\bar{u}}\ast(\hat{u}\ast \hat{u}) \right)\\
		=&\int_{\mathbb{R}^6}\frac{\hat{f}(\tau,\xi)\left< \xi\right>^{-1}\left< \tau+\xi^2\right>^a\hat{u}(\tau_1-\tau,\xi_1-\xi)\left< \xi_1-\xi\right>}{\left< \xi\right>^{-1}\left< \tau+\xi^2\right>^a\left< \xi_1-\xi\right>}\\
		& \frac{\left<\tau_1-\tau+(\xi_1-\xi)^2\right>^b\bar{\hat{u}}(\tau_1-\tau_2,\xi_1-\xi_2)\left<\xi_1-\xi_2\right>\left<\tau_1-\tau_2+(\xi_1-\xi_2)^2\right>^b}{\left<\tau_1-\tau+(\xi_1-\xi)^2\right>^b\left<\xi_1-\xi_2\right>\left<\tau_1-\tau_2+(\xi_1-\xi_2)^2\right>^b}\\
		&\frac{\bar{\hat{u}}(\tau_2,\xi_2)\left<\xi_2\right>\left<\tau_2+\xi_2^2\right>^b}{\left<\xi_2\right>\left<\tau_2+\xi_2^2\right>^b}  d\xi d\tau d\tau_1d\xi_1d\tau_2d\xi_2.
	\end{aligned}
	$$
	Let 
	$$
	\varphi(\tau,\xi)=\hat{f}(\tau,\xi)\left< \tau+\xi^2\right>^a,\  
	\bar{g}(\tau,\xi)=\hat{u}(-\tau,-\xi)\left<-\tau+\xi^2\right>^b,
	$$
	$$
	h(\tau,\xi)=\hat{u}(\tau,\xi)\left<\tau+\xi^2\right>^b,\  
	k(\tau,\xi)=\hat{u}(\tau,\xi)\left<\tau+\xi^2\right>^b.
	$$
	Because of $\left< \xi\right>\leq\left< \xi_1-\xi\right>\left<\xi_1-\xi_2\right>\left<\xi_2\right>$ and the Plancherel Theorem, it deduces to 
	\begin{equation}\label{trilinear_estimate_result_1}
		\begin{aligned}
			&\left(f,|u|^2u\right)\\
			\leq&C\int_{\mathbb{R}^2}  \Big|F^{-1}_{x,t}[\frac{\varphi(\tau,\xi)}{\left<\tau+\xi^2\right>^a}]\Big|\Big|F^{-1}_{x,t}[\frac{g(\tau,\xi)}{\left<-\tau+\xi^2\right>^b}]\Big| \Big|F^{-1}_{x,t}[\frac{h(\tau,\xi)}{\left<\tau+\xi^2\right>^b}]\Big|\\
			&\Big|F^{-1}_{x,t}[\frac{k(\tau,\xi)}{\left<\tau+\xi^2\right>^b}]\Big| dxdt\\
			\leq& C\Big\|F^{-1}_{x,t}[\frac{\varphi(\tau,\xi)}{\left<\tau+\xi^2\right>^a}]\Big\|_{L^4_{x,t}}\Big\|F^{-1}_{x,t}[\frac{g(\tau,\xi)}{\left<-\tau+\xi^2\right>^b}]\Big\|_{L^4_{x,t}}\Big\|F^{-1}_{x,t}[\frac{h(\tau,\xi)}{\left<\tau+\xi^2\right>^b}]\Big\|_{L^4_{x,t}}\\
			&\Big\|F^{-1}_{x,t}[\frac{k(\tau,\xi)}{\left<\tau+\xi^2\right>^b}]\Big\|_{L^4_{x,t}}\\
			=&C\Big\|F^{-1}_{x,t}[\frac{\varphi(\tau,\xi)}{\left<\tau+\xi^2\right>^a}]\Big\|_{L^4_{x,t}}\Big\|F^{-1}_{x,t}[\frac{g(-\tau,-\xi)}{\left<\tau+\xi^2\right>^b}]\Big\|_{L^4_{x,t}}\Big\|F^{-1}_{x,t}[\frac{h(\tau,\xi)}{\left<\tau+\xi^2\right>^b}]\Big\|_{L^4_{x,t}}\\
			&\Big\|F^{-1}_{x,t}[\frac{k(\tau,\xi)}{\left<\tau+\xi^2\right>^b}]\Big\|_{L^4_{x,t}}
		\end{aligned}
	\end{equation}
	By Lemma 2.9 in \cite{Tao2006}, we have that 
	\begin{equation}\label{trilinear_estimate_result_2}
		\begin{aligned}
			&\left(f,|u|^2u\right)\\
			\leq& C\Big\|F^{-1}_{x,t}[\frac{\varphi(\tau,\xi)}{\left<\tau+\xi^2\right>^a}]\Big\|_{X_{c,0}}\Big\|F^{-1}_{x,t}[\frac{g(-\tau,-\xi)}{\left<\tau+\xi^2\right>^b}]\Big\|_{X_{c,0}}\Big\|F^{-1}_{x,t}[\frac{h(\tau,\xi)}{\left<\tau+\xi^2\right>^b}]\Big\|_{X_{c,0}}\\
			&\Big\|F^{-1}_{x,t}[\frac{k(\tau,\xi)}{\left<\tau+\xi^2\right>^b}]\Big\|_{X_{c,0}}\\
			\leq&C\left\|\varphi\right\|_{L^2_{\tau,\xi}}\left\|g\right\|_{L^2_{\tau,\xi}}\left\|h\right\|_{L^2_{\tau,\xi}}\left\|k\right\|_{L^2_{\tau,\xi}},
		\end{aligned}
	\end{equation}
	for any $c\in(\frac{3}{8},a\wedge b) $.
	Hence, we finish the proof.
	
\end{proof}
 The next lemma is about the coupling term in KdV-type equation. 
 \begin{lemma}\label{lem:SS_to_K}
 	Suppose that $a\in(\frac{1}{4},\frac{1}{2})$, $b\in(\frac{1}{3},\frac{1}{2})$ and $a+2b>\frac{4}{3}$. Then for any $g,h\in X_{b,1}$, we have
 	\begin{equation}\label{SS_to_K_1}
 		\|\partial_x(g\bar{h})\|_{Y_{-a,1}}\leq C \| g\|_{X_{b,1}}\| h\|_{X_{b,1}}
 	\end{equation}
 	and 
 	\begin{equation}\label{SS_to_K_2}
 		\|\partial_x(g\bar{h})\|_{Y_{-a,1,-3/8}}\leq C \| g\|_{X_{b,1}}\| h\|_{X_{b,1}}.
 	\end{equation}
 \end{lemma}
 \begin{proof}
 	We first prove \eqref{SS_to_K_1}. By the duality, it is sufficient to prove that for $f,g,h\in\mathscr{S}(\mathbb{R}^2)$ 
 	$$
 	|(f,\partial_x(g\bar{h}))|\leq C \|f\|_{Y_{a,-1}}\|g\|_{X_{b,1}}\|h\|_{X_{b,1}}.
 	$$
 	By the Plancherel Theorem, we will estimate
 	\begin{equation}\label{fourier_form}
 		\begin{aligned}
 			\left|\int_{\mathbb{R}^4} \xi\hat{f}(\xi,\tau)\bar{\hat{g}}(\tau-\tau_1,\xi-\xi_1)\hat{h}(-\tau_1,-\xi_1)d\xi d\xi_1d\tau d\tau_1\right|.
 		\end{aligned}
 	\end{equation}
 	Just like before, we will divide $\mathbb{R}^4$ into 
    $$
    \begin{aligned}
    	&{\bf Reigon~I.}\ \   \left\{\mathbb{R}^4: \left< \tau-\xi^3\right>\geq 	E(\xi,\tau_1,\xi_1,\tau_1)\right\}, \\
    	&{\bf Reigon~II.}\   \left\{\mathbb{R}^4: \left<-\tau_1+\xi_1^2\right>\geq 	E(\xi,\tau_1,\xi_1,\tau_1)\right\}, \\
    	&{\bf Reigon~III.} \left\{\mathbb{R}^4: \left<\tau-\tau_1+(\xi-\xi_1)^2\right>\geq 	E(\xi,\tau_1,\xi_1,\tau_1)\right\}
    \end{aligned}
    $$
    and prove a stronger result, where $ E(\xi,\tau_1,\xi_1,\tau_1)$ is the maximum module. 
    
    We first note the fact 
    \begin{equation}\label{modify_and_E}
    	\begin{aligned}
    	3E(\xi,\tau_1,\xi_1,\tau_1)\geq \left<\xi(-\xi^2-\xi+2\xi_1) \right>
    	\end{aligned}
    \end{equation}

 	{\bf Region I.}: In this case, 	by Cauchy-Schwarz inequality, we only need to prove 
 	\begin{equation}\label{needed1}
 		\begin{aligned}
 			&\sup_{\mathbb{R}^2_{\xi,\tau}} \frac{|\xi|\left<\xi\right>}{\left<\tau-\xi^3\right>^a}\Big
 			(\int_{E\leq  \left< \tau-\xi^3\right>}\left<-\tau_1+\xi_1^2\right>^{-2b}\left<\tau-\tau_1+(\xi-\xi_1)^2\right>^{-2b} \\
 			&\qquad\qquad\cdot\left<\xi_1\right>^{-2}\left< \xi-\xi_1\right>^{-2}d\tau_1 d\xi_1\Big)^{1/2}
 			<\infty.
 		\end{aligned}
 	\end{equation}
 	By Lemma \ref{lem:basic_inequality_homogeneous}, it deduces that we need to prove 
 	\begin{equation}\label{integral_of_xi}
 		\begin{aligned}
 	 \sup_{\mathbb{R}^2_{\xi,\tau}}M(\xi,\tau):=\sup_{\mathbb{R}^2_{\xi,\tau}}\frac{|\xi|\left<\xi\right>}{\left<\tau-\xi^3\right>^a}\left(\int_{ E\leq \left< \tau-\xi^3\right>} \frac{\left<\xi_1\right>^{-2}\left< \xi-\xi_1\right>^{-2}}{\left<\tau+\xi^2+2\xi\xi_1\right>^{4b-1}} d\xi_1\right)^{\frac{1}{2}}<\infty.
 		\end{aligned}
 	\end{equation}
We set $\mu=\tau+\xi^2-2\xi\xi_1$.

   {\bf Case A.}$|\xi|\leq 100$. We directly have
 	$$
 	\sup_{\mathbb{R}^2_{\xi,\tau}}M(\xi,\tau)\leq C \frac{|\xi|^{1/2}}{\left< \tau-\xi^3\right>^a}\left(\int_{|\mu|\leq 2\left<\tau-\xi^3\right>}\frac{1}{\left<
 	\mu\right>^{4b-1}} d\mu\right)^{\frac{1}{2}}<\infty.
 	$$
 	
 	{\bf Case B.I.} $|\xi|>100$. $2|\xi_1|\leq \frac{1}{2}\xi^2$ or $ 2|\xi_1|\geq \frac{3}{2}\xi^2$. By \eqref{modify_and_E}, we have
 	$
 	\left<\tau-\xi^3\right>\geq C|\xi|^3
 	$. Therefore, we have 
 	$$
 	\begin{aligned}
    \sup_{\mathbb{R}^2_{\xi,\tau}}M(\xi,\tau)\leq \left<\tau-\xi^3\right>^{\frac{1}{6}-(2b+a-1)}< \infty
 	\end{aligned}
 	$$
 	because of $a+2b>\frac{7}{6}$.      
 	
 	{\bf Case B.II.} $|\xi|>100$. $\frac{1}{2}\xi^2<2|\xi_1|< \frac{3}{2}|\xi|^2$. Because of 
 	$
 	\left<\xi_1\right>^{-2}\left< \xi-\xi_1\right>^{-2}\leq C\left<\xi\right>^{-8}
 	$, 
 	we have
 	$$
 	\begin{aligned}
 		\sup_{\mathbb{R}^2_{\xi,\tau}}M(\xi,\tau)\leq C\left<\tau-\xi^3\right>^{-(2b+a-1)}<\infty.
 	\end{aligned}
 	$$

 	{\bf Region II.} Actually, in this region we can give a stronger proof than we need. We will estimate  
 	$$
 	\begin{aligned}
 		&\sup_{\mathbb{R}^2_{\xi_1,\tau_1}}M(\xi_1,\tau_1)\\
 		:=&\sup_{\mathbb{R}^2_{\xi_1,\tau_1}}\frac{1}{\left< -\tau_1+\xi_1^2\right>^b}\left(\int_{E\leq \left<-\tau_1+\xi_1^2\right>}\frac{\xi^2d\xi}{\left<\xi^3+(\xi-\xi_1)^2-\tau_1\right>^{2a+2b-1}}\right)^{1/2}
 	\end{aligned}
 	$$
 	
 	{\bf Case A.} $|\xi|\leq 100$. We directly have
 	$
 	\sup_{\mathbb{R}^2_{\xi,\tau}}M(\xi,\tau)<\infty.
 	$
 	
 	{\bf Case B.} $|\xi|>100$.
 	We set $\mu=\xi^3+(\xi-\xi_1)^2-\tau_1$. Then, we need to estimate 
 	$$
 	\begin{aligned}
 			&\sup_{\mathbb{R}^2_{\xi_1,\tau_1}}M(\xi_1,\tau_1)\\
 		=&\sup_{\mathbb{R}^2_{\xi_1,\tau_1}}\frac{1}{\left< -\tau_1+\xi_1^2\right>^b}\left(\int_{|\mu|\leq \left<-\tau_1+\xi_1^2\right>}\frac{\xi^2d\mu}{(3\xi^2+2\xi-2\xi_1)\left<\mu\right>^{2a+2b-1}}\right)^{1/2}.
 	\end{aligned}
 	$$

 	{\bf Case B.I.} $|\xi|>100$. $ |2\xi_1|\leq\frac{5}{2}\xi^2$ or $  |2\xi_1|\geq\frac{7}{2}\xi^2$.
    We have 
 	$$
 	\sup_{\mathbb{R}^2_{\xi_1,\tau_1}}M(\xi_1,\tau_1)\leq \sup_{\mathbb{R}^2_{\xi_1,\tau_1}}\frac{C}{\left< -\tau_1+\xi_1^2\right>^b}\left(\int_{|\mu|\leq 2\left< -\tau_1+\xi_1^2\right>}\frac{d\mu}{\left<\mu\right>^{2a+2b-1}}\right)^{1/2}<\infty.
 	$$
 	
 	{\bf Case B.II.} $|\xi|>100$ and $\frac{5}{2}\xi^2<|2\xi_1|<\frac{7}{2}\xi^2$. By \eqref{modify_and_E}, we have
 	 $
 	 \left< -\tau_1+\xi_1^2 \right>\geq C\left< \xi^3\right>\geq C|\xi|^3.
 	 $ Therefore, we have 
 	$$
 	\begin{aligned}
 		&\sup_{\mathbb{R}^2_{\xi_1,\tau_1}}M(\xi_1,\tau_1)\\
 		\leq& \frac{1}{\left<-\tau_1+\xi^2_1\right>^{b-1/3}}\left(\int_{E\leq \left<-\tau_1+\xi_1^2\right>}\frac{d\xi}{\left<\xi^3+(\xi-\xi_1)^2-\tau_1\right>^{2a+2b-1}}\right)^{1/2}
 		.
 	\end{aligned}
 	$$
 	
 	Let us further divide this case into following subcases.   
 	
 	{\bf Case B.II.1} $|\xi|>100$, $\frac{5}{2}\xi^2<|2\xi_1|<\frac{7}{2}\xi^2$ and $ |3\xi^2+2\xi-2\xi_1|\geq 1$.  Because of $ a+2b>4/3$, we have 
 	$$
 	\begin{aligned}
 		&\sup_{\mathbb{R}^2_{\xi_1,\tau_1}}M(\xi_1,\tau_1)\\
\leq& \frac{1}{\left<-\tau_1+\xi^2_1\right>^{b-1/3}}\left(\int_{|\mu|\leq \left<-\tau_1+\xi_1^2\right>}\frac{d\mu}{\left<\mu\right>^{2a+2b-1}}\right)^{1/2}<\infty.
 	\end{aligned}
 	$$

 	{\bf Case B.II.2}  $|\xi|>100$, $\frac{5}{2}\xi^2<|2\xi_1|<\frac{7}{2}\xi^2$ and $ |3\xi^2+2\xi-2\xi_1|< 1$. Thus, since the integral interval of $\xi$ is finite and $b>1/3$, we have 
 	$$
 	\begin{aligned}
 			&\sup_{\mathbb{R}^2_{\xi_1,\tau_1}}M(\xi_1,\tau_1)<\infty.
 	\end{aligned}
 	$$
 	
    {\bf Region III.} By the change of variable, this can be proved like region II.
    
    Now we consider the proof of \eqref{SS_to_K_2}. By duilty, for the case $|\xi|\geq1$, the proof is easier. Since in the proof of \eqref{SS_to_K_1}, the $|\xi|$ is controlled by $E(\tau,\tau_1,\xi,\xi_1)$ and we have not used the smallness of $|\xi|$, the proof of case $|\xi|<1$ is also clear. 
 \end{proof}

For any $T>0$, all the bilinear estimates and the trilinear estimate in the whole space can be restricted to $[0,T]$.
For example, we have
$$
\|\partial_x(g\bar{h})\|_{Y^T_{-a,1}}\leq C \| g\|_{X^T_{b,1}}\| h\|_{X^T_{b,1}}.
$$
This is because
$$
\begin{aligned}
	\|\partial_x(g\bar{h})\|_{Y^T_{-a,1}}\leq& \|\chi_{[0,T]}\partial_x(g\bar{h})\|_{Y_{-a,1}}=\|\partial_x(\chi_{[0,T]}g\chi_{[0,T]}\bar{h})\|_{Y_{-a,1}}\\
	\leq&C \| \chi_{[0,T]}(t)g\|_{X_{b,1}}\|\chi_{[0,T]}(t) h\|_{X_{b,1}}\\
	\leq&C(b) \|g\|_{X^T_{b,1}}\|h\|_{X^T_{b,1}},
\end{aligned}
$$
for any proper $a,b\in(0,1/2)$. Here, we have used the Lemma 2.1 in \cite{kdvmuti} to illustrate the last inequality.

Now, we derive the estimates of stochastic integral terms. According to the proof of Proposition 2.5 in \cite{kdvmuti}, we know the following Lemma. 

\begin{lemma}\label{lem:SI_1}
	Let $b\in[0,\frac{1}{2}) $, $T>0$, $l\in\mathbb{N}^+,\ k_1\in H_x^1$, $k_2\in H_x^1\cap L_x^1$, $u\in L^{2l}\left(\Omega; X^T_{b,1}\right)$ and $ v\in L^{2l}\left(\Omega; \tilde{Y}^T_{b,1}\right)$. Then, we have 
	\begin{equation}\label{SI_1}
		\begin{aligned}
			&\mathbb{E}\left\|\int_{0}^{t} S(t-r)\left(F(u)^{\alpha}\Phi dW^{(1)}_r\right)\right\|^{2l}_{X^{T\wedge\tau}_{b,1}}
			\leq C\left\| k_1\right\|_{H_x^1}^{2l}\mathbb{E}\|u^{\alpha}\|^{2l}_{X^{T\wedge\tau}_{0,1}},
		\end{aligned}
	\end{equation}
	\begin{equation}\label{SI_2}
	\begin{aligned}
		&\mathbb{E} \left\|\int_{0}^{t} U(t-r)\left(v^{\alpha}\Psi dW^{(2)}_r\right)\right\|^{2l}_{\tilde{Y}^{T\wedge\tau}_{b,1}}
		\leq C\left\| k_2\right\|_{H_x^1\cap L_x^1}^{2l}\mathbb{E}\|v^{\alpha}\|^{2l}_{Y^{T\wedge\tau}_{0,1}}
	\end{aligned}
\end{equation}
and 
\begin{equation}\label{stochastic_H_x^1H_x^{-3/8}}
	\begin{aligned}
		&\mathbb{E}  \left\|\int_{0}^{t} U(t-r)\left(v\Psi dW^{(2)}_r\right)\right\|_{L^\infty_{T\wedge \tau}H^1_x\cap\dot{H}^{-3/8}_x}^{2l}\\
		\leq &C\|k_2\|^{2l}_{H_x^1\cap L_x^1}T^l\mathbb{E} \|v\|_{L^\infty_{T\wedge \tau}H^1_x}^{2l},
	\end{aligned}
\end{equation}
for any stopping time $\tau$.
\end{lemma}
The proof of this lemma can be deduced directly from Proposition 2.5 and Proposition 2.7 in \cite{kdvmuti}. Furthermore, we should  estimate $\|u^{\alpha}\|^{2l}_{X^T_{0,1}} $ and $ \|v^{\alpha}\|^{2l}_{Y^T_{0,1}}$, for cases $\alpha\geq2$. 

\begin{lemma}\label{lem:strichartz_to_bourgain_2}
	For any $\alpha\in\mathbb{N^+}$,  we have 
	\begin{equation}\label{nonlinear_noise_esti_X_2}
			\|u^\alpha\|_{X^T_{0,1}}\leq C(\alpha) \|u\|^\alpha_{X^T_{b,1}}
	\end{equation}
	and
	\begin{equation}\label{nonlinear_noise_esti_Y_2}
	\|v^\alpha\|_{Y^T_{0,1}}\leq C(\alpha) \|v\|^\alpha_{Y^T_{b,1}},
	\end{equation}
where
	$$
b\in\left(b_\alpha,\frac{1}{2}\right):=\left(\left(\frac{1}{2}-\frac{1}{4(\alpha-1)}\right)\vee\frac{3}{8},\frac{1}{2}\right).
$$
\end{lemma}
\begin{proof}
	We prove the cases $\alpha\geq 2$. 
	
	Under the refined Strichartz estimates in \cite{chen}, we can get \eqref{nonlinear_noise_esti_X_2} and \eqref{nonlinear_noise_esti_Y_2} by the following manner: 
	
From the definition, we have
$$
\begin{aligned}
	&\|u^\alpha\|_{X^T_{0,1}}=\|u^\alpha\|_{L_T^2H_x^1}\\
	\leq& \|u^\alpha\|_{L^2_{T,x}} + \|u^{\alpha-1}\partial_{x}u\|_{L^2_{T,x}}\\
	\leq& \|u\|^\alpha_{L_{T,x}^{2\alpha}}+\|u\|^{\alpha-1}_{L_{T,x}^r}\|\partial_xu\|_{L_{T,x}^q}\\
	\leq &C(\alpha)\left(\|u\|^\alpha_{X^T_{b_1,1}}+\|u\|^{\alpha-1}_{X_{b_1,1}^T}\|\partial_x u\|_{X^T_{b_1,0}}\right)\\
	\leq&C(\alpha)\|u\|^\alpha_{X^T_{b_1,1}}.
\end{aligned}
$$
where, according to the result in \cite{chen}, we can choose $r=4(\alpha-1)$, $q=4$ and 
$$
b\in\left(\left(\frac{1}{2}-\frac{1}{4(\alpha-1)}\right)\vee\frac{3}{8},\frac{1}{2}\right).
$$
The fact $\|v^\alpha\|_{Y^T_{0,1}}\leq C(\alpha) \|v\|^\alpha_{Y^T_{b,1}}$ can get similarly.
\end{proof}

\section{The Local Solution}\label{sec:local-wellposedness}
In this section, for any $T_0>0$, we prove the global well-posedness of \eqref{truncated_eq}, which is equal to \eqref{mild_without_linear} in $\left[0,\sigma^{(1)}_R\wedge\sigma^{(2)}_R\wedge T_0\right]$ almost surely.

According to the proof of Lemma 2.2 in \cite{kdvmuti}, we can prove the following lemma of $\tilde{u}_R$ and $\tilde{v}_R$. 
\begin{lemma}\label{lem:truncated_equation}
For any $b\in(0,1/2)$, $R>0$, $u_R^{(i)}\in X_{b,1}^T $ and $v_R^{(i)}\in \tilde{Y}_{b,1}^T,\ i=1,2$, we have
$$
\left\|\tilde{u}^{(i)}_R(x,t)\right\|_{X^T_{b,1}}\leq C R,\ \left\|\tilde{v}^{(i)}_R(x,t)\right\|_{\tilde{Y}^T_{b,1}}\leq C R,
$$
$$
\left\|\tilde{u}^{(1)}_R(x,t)-\tilde{u}^{(2)}_R(x,t)\right\|_{X^T_{b,1}}\leq C \left\|u^{(1)}_R(x,t)-u^{(2)}_R(x,t)\right\|_{X^T_{b,1}} 
$$
and 
$$
\left\|\tilde{v}^{(1)}_R(x,t)-\tilde{v}^{(2)}_R(x,t)\right\|_{\tilde{Y}^T_{b,1}}\leq C \left\|v^{(1)}_R(x,t)-v^{(2)}_R(x,t)\right\|_{\tilde{Y}^T_{b,1}},
$$
where $C$ is independent to $R$.
\end{lemma}
\begin{proof}
	We only prove $\left\|\tilde{v}^{(1)}_R(x,t)\right\|_{\tilde{Y}^T_{b,1}}\leq C R$ as an example.
	
	This is because 
   $$
   \begin{aligned}
   	\|\tilde{v}^{(1)}_R(x,t)\|_{\tilde{Y}^T_{b,1}}=&\|\theta_R(\|v^{(1)}_R\|_{\tilde{Y}^t_{b,1}})v^{(1)}_R(x,t)\|_{\tilde{Y}^T_{b,1}}\\
   	=&\|\theta_R(\|J_x^1v^{(1)}_R\|_{\tilde{Y}^t_{b,0}})J_x^1v^{(1)}_R(x,t)\|_{\tilde{Y}^T_{b,0}}\\
   	\leq&CR,
   \end{aligned}
   $$
   where the last inequality comes from Lemma 2.2 in \cite{kdvmuti}. 
\end{proof}

According to the Lemma 2.9 in \cite{kdvmuti} or Lemma 3.11 and Lemma 3.12 in \cite{erdogan_tzirakis_2016}, we know that
\begin{equation}\label{restricted_norm_geq_1/2}
	\begin{aligned}
		&\left\|\int_{0}^{t} S(t-s) f  ds\right\|_{X^T_{b,1}}\leq C T^{1-(a+b)}\|f\|_{X^T_{-a,1}},\\
		&	\left\|\int_{0}^{t} U(t-s) g  ds\right\|_{\tilde{Y}^T_{b,1}}\leq C T^{1-(a+b)}\|g\|_{\tilde{Y}^T_{-a,1}},
	\end{aligned}
\end{equation}
for any $ T\in [0,1]$, $a,b\in(0,1)$ and $a+b<1$. \eqref{restricted_norm_geq_1/2} can be proved by the restricted norm method provided in \cite{erdogan_tzirakis_2016}.

Now, we can propose the proof of Theorem \ref{thm:local_well-posedness}.

\begin{proof}[{\bf Proof of Theorem \ref{thm:local_well-posedness}:}]
	We first prove $(u_R,v_R)\in L^{2l}\left(\Omega; X^T_{b,1}\times \tilde{Y}^T_{b,1}\right)$ through a fixed point argument. For $q=1,2$, we set
	$$
	\begin{aligned}
		&\mathscr{T}^{(1)}_R u^{(q)}_R(t)=S(t)u_0-i\int_{0}^{t} S(t-r)\left(\tilde{u}^{(q)}_R\left(\tilde{v}^{(q)}_R+U(r)w_0\right)+\left|\tilde{u}^{(q)}_R\right|^2\tilde{u}^{(q)}_R\right)dr\\
		&\qquad \qquad \qquad +\int_{0}^{t} S(t-r)\left(F\left(\tilde{u}^{(q)}_R\right)\right)^{\alpha}\Phi dW_r^{(1)},\\
		&\mathscr{T}^{(2)}_R v^{(q)}_R(t)=\int_{0}^{t} U(t-r)\partial_x\left(\left|\tilde{u}^{(q)}_R\right|^2-\frac{1}{2}\left(\tilde{v}^{(q)}_R+U(r)w_0\right)^2\right)dr\\
		&\qquad \qquad \qquad +\int_{0}^{t}U(t-r) \left(v^{(q)}_R+U(r)w_0\right)^{\alpha}\Psi dW_r^{(2)}.
	\end{aligned}
	$$
	For the contractility of $\mathscr{T}_R$ in $L^{2l}\left(\Omega; X^T_{b,1}\times \tilde{Y}^T_{b,1}\right)$, we deal with all the terms of $\mathscr{T}^{(1)}_R $ and the coupling term of $ \mathscr{T}^{(2)}_R$, while the proof of rest terms can be found in \cite{kdvmuti}.
	It follows from  Lemma \ref{lem:KS_to_S}-\ref{lem:trilinear_estimate}, Lemma \ref{lem:truncated_equation} and \eqref{restricted_norm_geq_1/2} that
	\begin{equation*}\label{contractility_SK_to_S_cubic}
		\begin{aligned}
		&\Big\|\int_{0}^{t} S(t-r)\Big(\tilde{u}^{(1)}_R\tilde{v}^{(1)}_R+\left|\tilde{u}^{(1)}_R\right|^2\tilde{u}^{(1)}_R-\tilde{u}^{(2)}_R\tilde{v}^{(2)}_R-\left|\tilde{u}^{(2)}_R\right|^2\tilde{u}^{(2)}_R\\
		&\qquad\qquad\qquad+\left( \tilde{u}^{(1)}_R-\tilde{u}^{(2)}_R \right)U(r)w_0\Big)dr\Big\|_{X^T_{b,1}}\\
		\leq&CT^{1-a-b}\Big( \left\|\tilde{u}_R^{(1)}-\tilde{u}_R^{(2)}\right\|_{X^T_{b,1}}\left\|\tilde{v}^{(1)}_R\right\|_{Y^T_{b,1}}+ \left\|\tilde{u}^{(2)}_R\right\|_{X^T_{b,1}}\left\| \tilde{v}_R^{(1)}-\tilde{v}_R^{(2)}\right\|_{Y^T_{b,1}}\\
		&\quad +\left\|\tilde{u}_R^{(1)}-\tilde{u}_R^{(2)}\right\|_{X^T_{b,1}}\left\|\tilde{u}^{(1)}_R\right\|^2_{X^T_{b,1}}+\left\|\tilde{u}_R^{(1)}-\tilde{u}_R^{(2)}\right\|_{X^T_{b,1}}\left\|\tilde{u}^{(2)}_R\right\|^2_{X^T_{b,1}}\\
		&\quad+\left\|\tilde{u}_R^{(1)}-\tilde{u}_R^{(2)}\right\|_{X^T_{b,1}}\left\|\tilde{u}^{(1)}_R\right\|_{X^T_{b,1}}\left\|\tilde{u}^{(2)}_R\right\|_{X^T_{b,1}}\\
		&\quad+C(T_0)\|w_0\|_{H_x^1}\left\|\tilde{u}_R^{(1)}-\tilde{u}_R^{(2)}\right\|_{X^T_{b,1}}\Big)\\
		\leq&CT^{1-a-b}\bigg(\left(C(T_0)\|w_0\|_{H_x^1}+R+R^2\right)\left\|\tilde{u}_R^{(1)}-\tilde{u}_R^{(2)}\right\|_{X^T_{b,1}}\\
		&\qquad+R \left\| v_R^{(1)}-v_R^{(2)}\right\|_{Y^T_{b,1}}\bigg),
		\end{aligned}
	\end{equation*}
for any $a\in(3/8,1/2)$, $b\in(b_\alpha,1/2)$. 

By Lemma \ref{lem:SI_1} and Lemma \ref{lem:truncated_equation}, we have 
\begin{equation*}
	\begin{aligned}
	&\left\|\int_{0}^{t} S(t-r)\left(F\left(\tilde{u}^{(1)}_R\right)^{\alpha}-F\left(\tilde{u}^{(2)}_R\right)^{\alpha}\right)\Phi dW_r^{(1)} \right\|_{L^{2l}_\omega X^T_{b,1}}\\
	\leq&C\mathbb{E}\left(\left\|\left(\tilde{u}^{(1)}_R\right)^{\alpha}-\left(\tilde{u}^{(2)}_R\right)^{\alpha}\right\|^{2l}_{X^T_{0,1}}\right)^{1/2l}\\
	\leq&C(\|k_1\|_{H_x^1}) T^{(b-b_{\alpha})/2}R^{\alpha-1}\left\|u^{(1)}_R-u^{(2)}_R\right\|_{L^{2l}_\omega X^T_{b,1}}
	\end{aligned}
\end{equation*}
and
\begin{equation*}
	\begin{aligned}
    &\left\|\int_{0}^{t}U(t-r) \left(\left(\tilde{v}^{(1)}_R+U(r)w_0\right)^{\alpha}-\left(\tilde{v}^{(2)}_R+U(r)w_0\right)^{\alpha}\right)\Psi dW_r^{(2)}\right\|_{L^{2l}_\omega \tilde{Y}^T_{b,1}}\\
	\leq&C(\|k_2\|_{H_x^1\cap L_x^1})T^{(b-b_{\alpha})/2}(R+C(T)\|w_0\|_{H_x^1})^{\alpha-1}\left\|v^{(1)}_R-v^{(2)}_R\right\|_{L^{2l}_\omega Y^T_{b,1}},
\end{aligned}
\end{equation*}
for any $b\in(b_\alpha,1/2)$. As for the coupling term in the KdV-type equation, Lemma \ref{lem:SS_to_K}, Lemma \ref{lem:truncated_equation} and \eqref{restricted_norm_geq_1/2} imply that
\begin{equation*}
	\begin{aligned}
		&\left\|\int_{0}^{t} U(t-r)\partial_x\left(\left|\tilde{u}^{(1)}_R\right|^2-\left|\tilde{u}^{(2)}_R\right|^2\right)dr\right\|_{\tilde{Y}_{b,1}^T}\\
		\leq&CT^{1-a-b}\left\|\partial_x\left(\left|\tilde{u}^{(1)}_R\right|^2-\left|\tilde{u}^{(2)}_R\right|^2\right) \right\|_{\tilde{Y}^T_{-a,1}}\\
		\leq&CT^{1-a-b}R\left\|u_R^{(1)}-u_R^{(2)}\right\|_{X^T_{b,1}},
	\end{aligned}
\end{equation*}
for any $a\in(3/8,1/2)$, $b\in(b_\alpha,1/2)$. Thus, by choosing $T^*=T^*(R)<1$ sufficiently small, we can get the contractility of $\mathscr{T}_R$ in $L^{2l}(\Omega; X^{T^*}_{b,1}\times \tilde{Y}^{T^*}_{b,1})$.   
Since the boundedness of $\mathscr{T}_R$ in $L^{2l}\left(\Omega; X^{T^*}_{b,1}\times \tilde{Y}^{T^*}_{b,1}\right)$ can be proved similarly, we can get the well-posedness of \eqref{truncated_eq} in $[0,T^*]$.

Then, we need to prove 
$$(u_R,v_R)\in L^{2l}\left(\Omega;C\left([0,T^*];H^1_x\times \left(H^1_x\cap \dot{H}_x^{-3/8}\right)\right)\right).
$$
For the nonlinear terms, by \eqref{restricted_norm_geq_1/2} we have
\begin{equation*}
	\begin{aligned}
		&\left\|\int_{0}^{t} S(-r)\left(\tilde{u}_R\left(\tilde{v}_R+U(r)w_0\right)+|\tilde{u}_R|^2\tilde{u}_R\right)dr\right\|_{L_{T^*}^\infty H_x^1}\\
		\leq&\left\|\int_{0}^{t} S(t-r)\left(\tilde{u}_R\left(\tilde{v}_R+U(r)w_0\right)+|\tilde{u}_R|^2\tilde{u}_R\right)dr\right\|_{X_{b',1}^{T^*}}\\
		\leq& C (T^*)^{1-(a+b')}\left\|\tilde{u}_R\left(\tilde{v}_R+U(r)w_0\right)+|\tilde{u}_R|^2\tilde{u}_R\right\|_{X^{T^*}_{-a,1}}\leq C(T^*,R)
	\end{aligned}
\end{equation*}
and
\begin{equation*}
	\begin{aligned}
		&\left\|\int_{0}^{t} U(-r)\partial_x\left(|\tilde{u}_R|^2-\frac{1}{2}\left(\tilde{v}_R+U(r)w_0\right)^2\right)dr\right\|_{L_{T^{*}}^\infty H_x^1\cap\dot{H}^{-3/8}_x}\\
		\leq &\left\|\int_{0}^{t} U(t-r)\partial_x\left(|\tilde{u}_R|^2-\frac{1}{2}\left(\tilde{v}_R+U(r)w_0\right)^2\right)dr\right\|_{\tilde{Y}_{b',1}^{T^*}}\\
		\leq &C (T^*)^{1-(a+b')}\left\|\partial_x\left(|\tilde{u}_R|^2-\frac{1}{2}\left(\tilde{v}_R+U(r)w_0\right)^2\right)\right\|_{\tilde{Y}^{T^*}_{-a,1}}\leq C(T^*,R,\|w_0\|_{H_x^1}),
	\end{aligned}
\end{equation*}
for any $b'>1/2$ satisfying $a+b'<1$. Here, we have used the important inequality
$$
\left\|\partial_x(f\cdot g)\right\|_{\tilde{Y}^{T}_{-a,1}} \leq C\min\{ \|f\|_{Y^T_{b',1}},\|f\|_{\tilde{Y}_{b,1}^T}\}\cdot \min\{\|g\|_{Y^T_{b',1}},\|g\|_{\tilde{Y}_{b,1}^T}\}
$$
in \cite{de1999white} to deal with $\left\|\partial_x(\tilde{v}_R\cdot U(r)w_0)\right\|_{\tilde{Y}^{T^*}_{-a,1}} $.

For the estimate of stochastic integrals, by the BDG inequality and Lemma 2.6 in \cite{kdvmuti}, we have 
\begin{equation*}
	\begin{aligned}
		&\left\|\int_{0}^{t} S(t-r)F(\tilde{u}_R)^{\alpha}\Phi dW_r^{(1)}\right\|^{2l}_{L_\omega^{2l}L_{T^*}^\infty H_x^1}\\
		\leq &C\|k_1\|^{2l}_{H_x^1}\mathbb{E}\left(\int_{0}^{T^*} \|\tilde{u}_R^{\alpha}\|^2_{H_x^1} dr\right)^l\\
		\leq &C(T^*)\|k_1\|^{2l}_{H_x^1}\mathbb{E}\left(\|\tilde{u}_R\|^{2\alpha}_{X^T_{b,1}}\right)^l\\
		\leq&C(T^*)\|k_1\|^{2l}_{H_x^1}R^{2 \alpha l}
	\end{aligned}
\end{equation*}
and 
\begin{equation*}
	\begin{aligned}
		&\left\|\int_{0}^{t} U(t-r)(\tilde{v}_R+U(r)w_0)^{\alpha}\Psi dW_r^{(2)}\right\|^{2l}_{L_\omega^{2l}L_{T^*}^\infty H_x^1\cap\dot{H}_x^{-3/8}}\\
		\leq &C(T^*)\left(\|k_2\|_{H_x^1}+\|k_2\|_{L_x^1}\right)^{2l}\left(R^{2l\alpha}+\|w_0\|^{2l\alpha}_{H_x^1}\right),
	\end{aligned}
\end{equation*}
for any $b\in(b_\alpha,1/2)$.

Since $S(t),U(t)$ are continue operators in $H_x^1\cap\dot{H}^{-3/8}_x$, we have proved  $(u_R,v_R)\in L^{2l}\left(\Omega;C\left([0,T^*];H^1_x\times \left(H^1_x\cap \dot{H}_x^{-3/8}\right)\right)\right)
$. 

Thus, we can get the global well-posedness of \eqref{truncated_eq} in $[0,T_0]$ by dividing $[0,T_0]$ into finite numbers of intervals shorter than $T^*$.

Hence, the proof of $(u_R,v_R)\in L^{2l}\left(\Omega;C\left([0,T_0];H_x^1\times \left(H_x^1\cap\dot{H}_x^{-3/8}\right)\right)\right)$  and $L^{2l}\left(\Omega; X^{T_0}_{b,1}\times \tilde{Y}^{T_0}_{b,1}\right)$ are finished.
\end{proof}

\section{The Global Solution}\label{sec:global_solution}
According to the structure of conservation laws of \eqref{approxiamation_eq}, we consider the global well-posedness of \eqref{msskdv} with $\alpha=1$. For the sake of simplicity, we choose $ F(u)= u$ in this section. Moreover, when we use the workspaces $X_{b,1}^T,\ X_{-a,1}^T$ or $\tilde{Y}_{b,1}^T,\ \tilde{Y}_{-a,1}^T$, $a,b$ fit the requirements of Theorem \ref{thm:local_well-posedness}.

Let $ \sigma^*=\lim_{R\uparrow \infty} \sigma_R^{(1)}\wedge\sigma^{(2)}_R$, $\mathbb{P}\text{-a.s.}$ In section \ref{sec:local-wellposedness}, we prove the well-posedness of $(u,v)$ in $\left[0, \sigma^{(1)}_R\wedge\sigma^{(2)}_R\wedge T_0\right]$ for any $R>0$. Thus, if we can prove 
\begin{equation}\label{final_aimming_proori_estimate}
	\mathbb{E}\|(u_R,v_R)\|^{2l}_{X^{\sigma_R^{(1)}\wedge\sigma^{(2)}_R\wedge T_0}_{b,1}\times \tilde{Y}^{\sigma_R^{(1)}\wedge\sigma^{(2)}_R\wedge T_0}_{b,1}} \leq C,
\end{equation}
for a constant $C$ independent to $R$, then we can prove $ \sigma^*\wedge T_0 = T_0 $, $\mathbb{P}\text{-a.s.}$ and finish the proof of Theorem \ref{thm:main}.

Since the deterministic S-KdV has the conservation structure in $H_x^1\times H_x^1$, one may deduce \eqref{final_aimming_proori_estimate} from an $H_x^1\times H_x^1$ estimate.

When we use conservation laws of deterministic S-KdV, we need to take derivatives of conserved quantities. These calculations need $(u_R(t),v_R(t))\in H_x^2\times H_x^3$. However, since $(u_0,v_0)\in H_x^1\times H_x^1$, the regularity of $(u_R,v_R)$ is a problem. 
To overcome it, we consider following approximation equations of \eqref{truncated_eq}, which are first introduced by \cite{chenguguo}:
\begin{equation}\label{approxiamation_eq}
	\left\{
	\begin{aligned}
		&d u_{m,n,K}=i\partial_{xx}u_{m,n,K}dt -i \psi_K(|u_{m,n,K}|^2)u_{m,n,K}w_{m,n,K}dt\\
		&\qquad\qquad-i|u_{m,n,K}|^2u_{m,n,K}\varphi_K(|u_{m,n,K}|^2)+u_{m,n,K} P_m\Phi dW_t^{(1)},\\
		&d w_{m,n,K}=-\partial_{xxx}w_{m,n,K}dt+  P_n\partial_x(\varphi_K(|u_{m,n,K}|^2)|u_{m,n,K}|^2)dt\\
		&\qquad\qquad-\frac{1}{2}P_n\partial_{x}(\varphi_K(w_{m,n,K})w_{m,n,K}^2)dt\\
		&\qquad\qquad+ P_m\left(w_{m,n,K}\Psi dW_t^{(2)}\right),\\
		&u_{m,n,K}(0)=P_mu_0(x),\  w_{m,n,K}(0)=P_mw_0(x).
	\end{aligned}
	\right.
\end{equation}
Here, $n\geq m$, $\varphi\in C_0^\infty$ is a real cut-off function satisfying  $\varphi|_{[-1,1]} = 1$ and 
$$\varphi_K(x) = \varphi(x/K),\ \psi_K(x) = x\varphi_K'(x)+\varphi_K(x).$$

Because the nonlinear terms in \eqref{approxiamation_eq} are cut-off in both physical and frequency space, this kind of approximation equation can be proved that is well-posed in a high-regular space until $T_0$. 

Moreover, from \cite{chenguguo}, we know that the corresponding deterministic equation of  \eqref{approxiamation_eq} has three conservation laws:
$$
\|u_{m,n,K}\|^2_{L_x^2},\ 
I_{t}(u_{m,n,K},w_{m,n,K}): = \int_{\mathbb{R}}\text{Im}(u_{m,n,K}\partial_x\bar{u}_{m,n,K}) + \frac{1}{2}w_{m,n,K}^2 dx
$$
and 
\begin{equation*}
	\begin{aligned}
		&\quad\mathcal{E}_t(u_{m,n,K},w_{m,n,K})\\
		&:=\int_{\mathbb{R}}|\partial_x u_{m,n,K}|^2+\frac{1}{2}(|\partial_x w_{m,n,K}|^2-\psi_{2,K}(w_{m,n,K}))\\
		&\quad+\varphi_K(|u_{m,n,K}|^2)|u_{m,n,K}|^2w_{m,n,K}
		+\psi_{1,K}(|u_{m,n,K}|^2)~dx,
	\end{aligned}
\end{equation*}
where $$\psi_{1,K}(x) = \int_0^x s\varphi_K(s)~ds,\quad\psi_{2,K}(x) = \int_0^x s^2\varphi_K(s)~ds.$$

These conservation laws will help us get  $H_x^1\times H_x^1$ and the Bourgain norm priori estimate of \eqref{approxiamation_eq}.

\begin{prop}\label{prop:high_globel_well-posedness}
 For any $ T_0, R>0$, there exists a unique solution in $C([0, T_0]; H_x^2\times H_n)\ \mathbb{P}\text{-a.s.}$ for \eqref{approxiamation_eq}, where $$H_n=\{h\in L^2(\mathbb{R}):\mathrm{supp}~\mathcal{F}(h)\subset [-n,n]\}.$$
\end{prop}
\begin{proof}
	For the sake of simplicity, we use $u, w$ to replace $ u_{m,n,K}, w_{m,n,K}$ in this proof.
	Let 
	$$\|u\|_{\mathscr{X}_T}:=\|u\|_{C([0,T];L_x^2)}+\|u\|_{L_{x,T}^6},\quad \|w\|_{\mathscr{Y}_{T}}:=\|w\|_{C([0,T];H_n)}+\|w\|_{{L_{x,T}^8}}.$$
	We also introduce the following notations and definitions
	$$
	\breve{u}_\lambda:=\theta_\lambda(\|u_\lambda\|_{\mathscr{X}_t})u_\lambda,\ \breve{w}_\lambda:=\theta_\lambda(\|w_\lambda\|_{\mathscr{Y}_{t}})w_\lambda,
	$$
	$$
	\tau_\lambda^{(1)}:=\inf\{t\geq 0: \|u_\lambda\|_{\mathscr{X}_t} \geq \lambda\},\
	\tau_\lambda^{(2)}:=\inf\{t\geq 0: \|w_\lambda\|_{\mathscr{Y}_{t}} \geq \lambda\}
	$$
	$$
	\tau_\lambda:=\tau_\lambda^{(1)}\wedge\tau_\lambda^{(2)},\ \forall \lambda \geq 0
	$$
	and consider the localized equation:
	\begin{equation}\label{approxiamation_eq_localized}
		\left\{
		\begin{array}{rcl}
			d u_\lambda&=&i\partial_{xx}u_{\lambda}dt -i \psi_K\left(|\breve{u}_{\lambda}|^2\right)\breve{u}_{\lambda}\breve{w}_{\lambda}dt -i|\breve{u}_{\lambda}|^2\breve{u}_{\lambda}\varphi_K\left(|\breve{u}_{\lambda}|^2\right)dt\\ 
			&&+u_{\lambda}P_m\Phi dW_t^{(1)},\\
			d w_{\lambda}&=&-\partial_{xxx}w_\lambda dt+  P_n\partial_x\left(\varphi_K(|\breve{u}_{\lambda}|^2)|\breve{u}_{\lambda}|^2\right)dt\\
			&&-\frac{1}{2}P_n\partial_{x}\left(\varphi_K(\breve{w}_{\lambda})\breve{w}_{\lambda}^2\right)dt+ P_m\left(w_{\lambda}\Psi dW_t^{(2)}\right),\\
			u_\lambda(0)&=&P_mu_0(x),\  w_\lambda(0)=P_mw_0(x).
		\end{array}
		\right.
	\end{equation}

It is not difficult to prove that 
$$
\|\breve{u}_{\lambda}\|_{\mathscr{X}_T}\leq C\cdot \lambda,\  \|\breve{w}_{\lambda}\|_{\mathscr{Y}_T}\leq C\cdot \lambda
$$
and 
$$
\|\breve{u}^{(1)}_{\lambda}-\breve{u}^{(2)}_{\lambda}\|_{\mathscr{X}_T}\leq C\|u^{(1)}_{\lambda}-u^{(2)}_{\lambda}\|_{\mathscr{X}_T},\  \|\breve{w}^{(1)}_{\lambda}-\breve{w}^{(2)}_{\lambda}\|_{\mathscr{Y}_T}\leq C\|w^{(1)}_{\lambda}-w^{(2)}_{\lambda}\|_{\mathscr{Y}_T}.
$$
Therefore, by the proof of Proposition 4.1 in \cite{chenguguo} and Corollary 3.1 in \cite{de1999stochastic}, to do a fixed point argument in $ L^8\left(\Omega;\mathscr{X}_{ T}\times\mathscr{Y}_{T}\right),\ \forall T>0$, we have
\begin{equation}\label{contractility_approximation_1}
	\begin{aligned}
		&\left\|u^{(1)}_{\lambda}-u^{(2)}_{\lambda}\right\|^8_{L^8\left(\Omega;\mathscr{X}_{T}\right)}\\
		\leq &C_1(T)\left(\lambda^8+\lambda^{24}\right)\left\|u^{(1)}_{\lambda}-u^{(2)}_{\lambda}\right\|^8_{L^8\left(\Omega;\mathscr{X}_{T}\right)} +C_1(T) \lambda^8\left\|w^{(1)}_{\lambda}-w^{(2)}_{\lambda}\right\|^8_{L^8\left(\Omega;\mathscr{Y}_{T}\right)} \\
		&+C_1(T)\|k_1\|^8_{H_x^1}\left\|u^{(1)}_{\lambda}-u^{(2)}_{\lambda}\right\|^8_{L^8\left(\Omega;\mathscr{X}_{T}\right)}\\
	\end{aligned}
\end{equation}
and
\begin{equation}\label{contractility_approximation_2}
	\begin{aligned}
		&\left\|w^{(1)}_{\lambda}-w^{(2)}_{\lambda}\right\|^8_{L^8\left(\Omega;\mathscr{Y}_{T}\right)}\\
		\leq &C_2(T)n^8(\lambda^8+\lambda^{16})\left\|u^{(1)}_{\lambda}-u^{(2)}_{\lambda}\right\|^8_{L^8\left(\Omega;\mathscr{X}_{T}\right)} \\
		&+C_2(T)n^8(\lambda^8+\lambda^{16})\left\|w^{(1)}_{\lambda}-w^{(2)}_{\lambda}\right\|^8_{L^8\left(\Omega;\mathscr{Y}_{T}\right)} \\
		&+C_2(T)\|k_2\|^8_{H_x^1}\left\|w^{(1)}_{\lambda}-w^{(2)}_{\lambda}\right\|^8_{L^8\left(\Omega;\mathscr{Y}_{T}\right)}.
	\end{aligned}
\end{equation}
Here, $C_1(T),\ C_2(T)$ will decrease to $0$ as $T\downarrow0$. We also used the fact  $\|P_mf\|_{L_x^p}\leq \|f\|_{L_x^p},\ \forall p>1$, (see for example \cite{classical_fourier_analysis}) in the proof of \eqref{contractility_approximation_1} and \eqref{contractility_approximation_2}. 

Therefore, for any fixed $ \lambda>0$, if we choose $T>0$ sufficiently small, we can get a local solution in 
$$ L^8\left(\Omega;\mathscr{X}_{ T}\times\mathscr{Y}_{T}\right).$$

Furthermore, if we divide $[0,T_0]$ into finite intervals, we can get the global well-posedness of \eqref{approxiamation_eq_localized} in 
$$ L^8\left(\Omega;\mathscr{X}_{ T_0}\times\mathscr{Y}_{T_0}\right).$$

According to the Strichartz estimate, the Leibniz-type estimate in \cite{benea2016multiple}, the G-N inequality and $\text{supp}\  \hat{w}_\lambda $ is compact, for any $ T,\varepsilon>0$, we can improve the regularity of $u_\lambda$  by the following inequality:
\begin{equation}\label{H_x^3_regulity_of _u_lambda}
	\begin{aligned}
		&\|u_\lambda\|^8_{L_\omega^8 L^\infty_{T}H_x^2}\\
		\leq& \left<m\right>^{16}\|u_0\|^8_{L_x^2} +\left\|  J^2\left[\left(\varphi'_K\left(|\breve{u}_{\lambda}|^2\right)|\breve{u}_{\lambda}|^2+\varphi_K\left(|\breve{u}_{\lambda}|^2\right)\right)\breve{u}_{\lambda}\breve{w}_{\lambda}\right]\right\|^8_{L_\omega^8L_{x,T}^{6/5}} \\
		&+\left\|  J^2\left[\varphi_K\left(|\breve{u}_{\lambda}|^2\right)|\breve{u}_{\lambda}|^2\breve{u}_{\lambda}\right]\right\|^8_{L_\omega^8L_{x,T}^{6/5}}\\
		&+C\left<m\right>^{16}\|k_1\|_{H_x^1}^8\mathbb{E}\left(\int_{0}^{T}\|u_\lambda\|^2_{H_x^2}ds\right)^4\\
		\leq&\left<m\right>^{16}\|u_0\|^8_{L_x^2}+ C(T)\left<m\right>^{16}\|k_1\|_{H_x^1}^8\mathbb{E}\int_{0}^{T}\|u_\lambda\|^8_{H_x^2}ds\\
		&+C(T)\left(K^8\lambda^{16}+K^8\lambda^8+K^{16}\lambda^8+\lambda^8+\lambda^8\left<m\right>^{16}\right)\|u_\lambda\|^8_{L_\omega^8 L^\infty_{T}H_x^2}\\
		&+ C(T)(mK\lambda)^{16}+C(T)C(\varepsilon)K^{216}\lambda^{72}+\varepsilon\|u_\lambda\|^8_{L_\omega^8 L^\infty_{T}H_x^2}.
	\end{aligned}
\end{equation}

Therefore, if we choose $T,\varepsilon$ sufficiently small and divide $[0,T_0]$ into finite numbers of small intervals, we will have 
\begin{equation}\label{H_x^3_well-posedness}
\|u_\lambda\|^8_{L_\omega^8 L^\infty_{T_0}H_x^2}\leq C\left(m,T_0,K,\lambda,\|k_1\|_{H_x^1},\|u_0\|_{L_x^2}\right),
\end{equation}
for any $\lambda>0$.

Let $\tau^*:=\lim_{\lambda\uparrow\infty}\tau_{\lambda}\ \mathbb{P}\text{-a.s.}$
To prove $  T_0\wedge \tau^*= T_0,\ \mathbb{P}\text{-a.s.}$, we need a priori estimate. For any $\lambda>0$, let 
$$
\begin{aligned}
H_\lambda(t)=\|u_\lambda(t)\|^8_{L_x^2}+\|w_\lambda(t)\|^8_{L_x^2}.
\end{aligned}
$$
Similar to \eqref{H_x^3_regulity_of _u_lambda}, we have
\begin{equation*}
	\begin{aligned}
			&\mathbb{E}\|H_\lambda(t)\|_{L^\infty_{ T_0\wedge\tau_{\lambda}}}\\
			\leq&H_\lambda(0)+  C(T_0)\cdot\left(\|k_1\|_{H_x^1}^8+C(n,K)
			\right)\mathbb{E}\int_{0}^{T_0\wedge\tau_{\lambda}} \|u_\lambda\|^8_{L_x^2} dt\\
			&+C(T_0)\cdot \left(\|k_2\|_{H_x^1}^8+C(n,K)\right) \mathbb{E}\int_{0}^{ T_0\wedge\tau_{\lambda}} \|w_\lambda\|^8_{L_x^2} dt\\
			\leq& H_\lambda(0)+ C(K,n,\|k_1\|_{H_x^1},\|k_2\|_{H_x^1})\int_{0}^{T_0}\mathbb{E}\|H_\lambda(s)\|_{L^\infty_{ t\wedge\tau_{\lambda}}}dt.
		\end{aligned}
\end{equation*}
Therefore, by the Gronwall inequality, we have 
\begin{equation}\label{L_x^2_priori_estimate}
	\begin{aligned}
			&\mathbb{E}\|H_\lambda(t)\|_{L^\infty_{ T_0\wedge\tau_{\lambda}}}\\
			\leq& C(K,n,\|k_1\|_{H_x^1},\|k_2\|_{H_x^1},\|u_0\|_{L_x^2},\|w_0\|_{L_x^2})e^{C(K,n,\|k_1\|_{H_x^1},\|k_2\|_{H_x^1})T_0}.
		\end{aligned}
\end{equation}
Let $Q_\lambda(t):= \|u_\lambda\|^8_{L^6_{ t\wedge\tau_{\lambda}}L_x^6}+\|w_\lambda\|^8_{L^8_{ t\wedge\tau_{\lambda}}L_x^8}$ . The Strichartz estimate, Lemma 3.1, Corollary 3.1  in \cite{de1999stochastic} and the H\"{o}lder inequality imply
$$
\begin{aligned}
	&\mathbb{E}Q_\lambda(T_0)\\
	\leq& C\|u_0\|^8_{L_x^2}+C\|w_0\|^8_{L_x^2} +C(n,K,T_0)\mathbb{E}\int_{0}^{ T_0\wedge\tau_{\lambda}}\|u_{\lambda}\|^8_{L_x^2}+\|w_{\lambda}\|^8_{L_x^2} ds\\
	&+CT_0^{1/3}\|k_1\|^8_{H_x^1}\int_{0}^{T_0}\mathbb{E}\left(\int_{0}^{ T_0\wedge\tau_{\lambda}}|t-s|^{-2/3}\|u_{\lambda}\|^2_{L_x^2} ds\right)^4dt\\
	&+C\|k_2\|^8_{H_x^1}\int_{0}^{T_0}\mathbb{E}\left(\int_{0}^{ T_0\wedge\tau_{\lambda}}|t-s|^{-1/2}\|w_{\lambda}\|^2_{L_x^2} ds\right)^4dt\\
	\leq
	&C(n,K,\|k_1\|_{H_x^1},\|k_2\|_{H_x^1},T_0)\mathbb{E}\int_{0}^{ T_0\wedge\tau_{\lambda}}\|u_{\lambda}\|^8_{L_x^2}+\|w_{\lambda}\|^8_{L_x^2} ds,
\end{aligned}
$$ 
which means
\begin{equation}\label{L_x,t^6}
	\mathbb{E}Q_\lambda(T_0)\leq C(K,n,\|k_1\|_{H_x^1},\|k_2\|_{H_x^1},\|u_0\|_{L_x^2},\|w_0\|_{L_x^2},T_0).
\end{equation}

Now, combing all the inequalities above, we can finish this proof. 
Since \eqref{L_x^2_priori_estimate} and \eqref{L_x,t^6} are independent to $\lambda$, we have $  T_0\wedge \tau^*=  T_0,\ \mathbb{P}\text{-a.s.}$

Moreover, because of $ (u(t),w(t))=(u_\lambda(t),w_\lambda(t)),\ t\in\left[0, \tau_\lambda\right],\ \mathbb{P}\text{-a.s.}$,
we have
$$
(u,w)\in 
C\left(\left[0,  T_0\right]; H_x^2\times H_n\right),\ \mathbb{P}\text{-a.s.}
$$

\end{proof}

The following remark is about the reason why our approximation equations \eqref{approxiamation_eq} are necessary.
\begin{rem}\label{rem:irreplaceability}
	In general, to use the conservation laws of \eqref{truncated_eq}, one may firstly smooth the noises and the initial datum to get the $H_x^2\times H_x^3$ local solution. 
	Then, by some stopping time skills, one can get a solution in $C([0, \sigma^{(1)}_R\wedge\sigma^{(2)}_R\wedge T_0\wedge \tilde{\tau}^*];H_x^2\times H_x^3)\ \mathbb{P}\text{-a.s.}$, where $\lim_{t\uparrow\tilde{\tau}^*}\|(u_R,v_R)\|_{H_x^2\times H_x^3}=\infty$. 
	However, although we can get a priori estimate of $(u_R,v_R)$ in $L^{2l}(\Omega;C([0,\sigma^{(1)}_R\wedge\sigma^{(2)}_R\wedge T_0\wedge \tilde{\tau}^*];H_x^1\times H_x^1))$ by conservation laws, it is still difficult to illustrate $\sigma^{(1)}_R\wedge\sigma^{(2)}_R\wedge T_0\wedge \tilde{\tau}^*=\sigma^{(1)}_R\wedge\sigma^{(2)}_R\wedge T_0 \ \mathbb{P}\text{-a.s.}$, since we do not have $\|u_R\|_{L_x^2L_T^\infty},\ \|v_R\|_{L_x^2L_T^\infty}$ to deal with terms like
	$$
	\begin{aligned}
	\left\|\int_{0}^{t}  U(t-s)\partial_{x} \left(|u_R|^2\right) ds\right\|_{L_T^\infty H_x^2}\leq 2\|u_R\|_{L_x^2L_T^\infty}\left\|J^2u_R\right\|_{L_T^2L_x^2}.
	\end{aligned}
	$$ 
	In another words, the inequality 
	\begin{equation}\label{appendix_inequality}
		\|\cdot\|_{L_x^2L_T^\infty }\leq \|\cdot\|_{L_T^\infty H_x^1}
	\end{equation}
	is not true. Intuitively speaking, this fact will let us can not do the $H_x^1\times H_x^1$ priori estimate in the whole $[0,\sigma^{(1)}_R\wedge\sigma^{(2)}_R\wedge T_0] $ interval.
	For the proof of the falseness of \eqref{appendix_inequality}, we put it in the appendix.
	
	Moreover, when we prove the global well-posedness of the approximation equations in $[0,T_0]$, we only use the Strichartz estimate, which means these approximation equations
	can help us avoid proving some complicated multilinear estimates and localized Bourgain norm estimates.
\end{rem}

\begin{prop}\label{prop:uniform_esti} 
We still use $(u,w)$ to represent $ (u_{m,n,K},w_{m,n,K})$. Suppose that $(u_0,w_0)\in  H_x^1\times H_x^1$ and $k_1,k_2\in H_x^1$, then we have
\begin{equation}\label{priori_esti}
 \mathbb{E}\left(\|(u,w)\|^{2l}_{L_{T_0}^\infty \mathcal{H}_x^1}\right)\leq C\left(\|u_0\|_{H_x^1}, \|w_0\|_{H_x^1},\|k_1\|_{H_x^{1}},\|k_2\|_{H_x^{1}},T_0,l\right),
\end{equation}
for any $T_0>0$, $l\in\mathbb{N^+}$.
\end{prop}
\begin{proof}
Let us consider the conservation laws of the corresponding deterministic equation of \eqref{approxiamation_eq}.

By Ito's formula, we have
\begin{equation*}
	\mathbb{E}\|u(t)\|^{2l}_{L^\infty_{T}L_x^2}
	\leq \|u(0)\|^{2l}_{L_x^2}+C(l,\|k_1\|_{H_x^1},T)\mathbb{E}\int_{0}^{T} \|u(t)\|^{2l}_{L_x^2} dt,
\end{equation*}
for any  $T>0,\ l\in \mathbb{N^+}$. Then, the Gronwall inequality implies
\begin{equation}\label{conservation_L_x^2}
\mathbb{E}\|u(t)\|^{2l}_{L^\infty_{T}L_x^2}
\leq C(l,\|k_1\|_{H_x^1}, T, \|u_0\|_{L_x^2}).
\end{equation}

For $I_{t}(u,w)$, we have
\begin{equation}\label{conservation-I}
	\begin{aligned}
	 &I_{t}(u,w)\\
	  =&I_0(u,w)+\text{Im} \int_{0}^{t}\left(u P_m\Phi dW_s^{(1)},\partial_x u\right)+\text{Im} \int_{0}^{t}\left(u,\partial_{x}\left(u P_m\Phi dW_s^{(1)}\right)\right)\\
	  &+\frac{1}{2}\sum_{k=0}^{\infty}\text{Im} \int_{0}^{t}\left(u P_m\Phi e_k,\partial_x\left(uP_m\Phi e_k\right)\right) ds\\
	  &+\int_{0}^{T_0}\left(w,P_m(w\Psi) dW_s^{(2)}\right)+\frac{1}{2}\int_{0}^{T_0}\|P_m(w\Psi)\|^2_{L_2^{0,0}}ds,
	\end{aligned}
\end{equation}
almost surely, for any $ t>0$.

It can be deduced by \eqref{conservation_L_x^2}, \eqref{conservation-I} that 
\begin{equation}\label{conservation-I_bound}
	\begin{aligned}
	&\mathbb{E}\|I_t(u,w)\|^q_{L^\infty_{T}}\\
	\leq &C(\|u_0\|_{H_x^1}, \|w_0\|_{L_x^2})+ C(q)T^{q/2}\|k_2\|^q_{L_x^2}\mathbb{E}  \| w\|^{2q}_{L^\infty_{T}L_x^2}\\
	&+C(q)T^{q/2}\|k_1\|^{q}_{L_x^2}\mathbb{E}\left(\|u\|^{q}_{L^\infty_{T}L_x^2}\|u\|^{q}_{L^\infty_{T}H_x^1}\right),
	\end{aligned}
\end{equation}
for any $q\geq1$.

Moreover, for $\mathcal{E}_{t}(u,w)$ we have

\begin{equation}\label{energy}
	\begin{aligned}
		&\mathcal{E}_{t}(u,w)\\
		=&\mathcal{E}_0(u,w)+\int_{0}^{t} \left(2\partial_x \text{Re}u,\partial_x \left(u P_m\Phi dW_s^{(1)}\right)\right) \\
		&+\sum_{k=0}^{\infty}\int_{0}^{t}  \left(\partial_{x}\left(uP_m\Phi e_k\right),\partial_{x}\left(uP_m\Phi e_k\right)\right)ds +\int_{0}^{t} \left(\partial_x w,\partial_xP_m\left(w\Psi dW_s^{(2)}\right)\right) \\
		&+\frac{1}{2}\sum_{k=0}^{\infty}\int_{0}^{t}\left(\partial_xP_m(w\Psi e_k),\partial_xP_m(w\Psi e_k)\right)ds \\
		&+\int_{0}^{t} \left(\varphi_K(|u|^2)|u|^2-\frac{1}{2} w^2\varphi_K(w),P_m\left(w\Psi dW_s^{(2)}\right)\right)\\
		&-\frac{1}{4}\sum_{k=0}^{\infty}\int_{0}^{t} \left(2w\varphi_K(w)+K^{-1}w^2\varphi_K'(w),(P_m(w\Psi e_k))^2\right) ds \\
		&+\int_{0}^{t} \left(\varphi'_K\left(|u|^2\right)|u|^2K^{-1}+\varphi_K\left(|u|^2\right), 2w\text{Re} u\cdot u P_m\Phi dW_s^{(1)}\right)\\
		&+\sum_{k=0}^\infty \int_{0}^{t} \left( \varphi_K(|u|^2),w(uP_m\Phi e_k)^2\right)ds\\
		&+\sum_{k=0}^\infty \int_{0}^{t} \left( K^{-1}\varphi'_K(|u|^2),w(|u|^2+\text{Re}(u^2))(u P_m\Phi e_k)^2\right)ds\\
		&+\sum_{k=0}^\infty \int_{0}^{t} \left( K^{-2}\varphi''_K(|u|^2),w|u|^2(|u|^2+\text{Re}(u^2))(u P_m\Phi e_k)^2\right)ds\\
		&+2\int_{0}^{t}\left( |u|^2\varphi_K\left(|u|^2\right),\text{Re}u\cdot u\cdot P_m\Phi dW_s^{(1)}\right)\\
		&+ 2\sum_{k=0}^{\infty}\int_{0}^{t} \left(\varphi_K\left(|u|^2\right)+K^{-1}|u|^2\varphi'_K\left(|u|^2\right),(\text{Re}u\cdot u P_m\Phi e_k)^2 \right)ds,
	\end{aligned}
\end{equation}
almost surely for any $t\in[0,T_0]$. 

It can be proved easily that there exists positive constants $C$ such that
\begin{equation}\label{lowerbound_of_conservation}
	\begin{aligned}
		Q_t(u,w):=&\left\|u\left(t\right)\right\|^2_{L_x^2}+\left\|u\left(t\right)\right\|^{10}_{L_x^2}\\
		&+\left|I_{t}(u,w)\right|+\left|I_{t}(u,w)\right|^{\frac{5}{3}}+\left|\mathcal{E}_{t}(u,w)\right|\\
		\geq& C\Big(\left\|u\left(t\right)\right\|^2_{H_x^1}+\left\|w\left(t\right)\right\|^2_{H_x^1}+\left\|u\left(t\right)\right\|^4_{L_x^4}+\left\|w\left(t\right)\right\|^3_{L_x^3}+ \left\| w\left(t\right)\right\|^{10/3}_{L_x^2}\Big),
	\end{aligned}
\end{equation}
for any $t\in[0,T_0]$.

For $Q_t(u,w)$, by \eqref{conservation_L_x^2}, \eqref{energy} and  \eqref{lowerbound_of_conservation},
we have 
\begin{equation*}
	\begin{aligned}
		&\mathbb{E}\left\|\left(u(t\right),w\left(t\right)\right\|^{2l}_{L^\infty_{T}\mathcal{H}_x^1}\leq \mathbb{E}\sup_{t\in[0,T]}Q^l_t(u,w)\\
		\leq& C\left(\|u_0\|_{H_x^1}, \|w_0\|_{H_x^1},\|k_1\|_{H_x^{1}},\|k_2\|_{L_x^2},T_0,l\right)\\
		&+\left(C(l)T^{l/2}\|k_2\|^l_{L_x^2}+C\right)\mathbb{E}\left(\|u\|^l_{L^\infty_{T}L_x^2}\|\partial_{x}u\|^l_{L^\infty_{T}L_x^2}\right)\\
		&+\left(C(l)T^{l/2}\|k_2\|^l_{L_x^2}+C\right)\mathbb{E}\left(\|u\|^{5l/3}_{L^\infty_{T}L_x^2}\|\partial_{x}u\|^{5l/3}_{L^\infty_{T}L_x^2}\right)\\
		&+C\left(\|k_1\|^{2l}_{H_x^{1}}+\|k_2\|^{2l}_{H_x^1}\right)T^{l-1}\mathbb{E}\int_{0}^{T}\|u\|_{H_x^1}^{2l}+\|w\|_{H_x^1}^{2l} ds\\
		&+C\|k_2\|^{l}_{L_x^{2}}\left(T^{2l-3}\mathbb{E}\int_{0}^{T} \|u\|^{2l}_{H_x^1}ds+\mathbb{E}\int_{0}^{T} \|w\|^{2l}_{L_x^2}ds\right)\\
		&+C\left(\|k_2\|^l_{L_x^2}+\|k_2\|^{2l}_{L_x^2}\right)\Bigg(T^{l/2}\mathbb{E}\|w\|^{10l/3}_{L^\infty_{T}L_x^2} +T^{l}\mathbb{E}\|w\|^{3l}_{L^\infty_{T}L_x^3}+T^l\mathbb{E}\|w\|_{L^\infty_{T_0}H_x^1}^{2l} \Bigg)\\
		&+C\|k_1\|^{2l}_{L_x^2}\left(T^{l-1}\mathbb{E}\int_{0}^{T} \|w\|^{2l}_{H_x^1}ds+  T^l\mathbb{E}\|u\|^{4l}_{L^\infty_{T}L_x^2}\right)\\
		&+C\left(\|k_1\|^{l}_{L_x^2}+\|k_1\|^{2l}_{L_x^2}\right)\Bigg(T^l\mathbb{E}\|u\|_{L_{T}^\infty L_x^2} +T^{l-1}\mathbb{E}\int_{0}^{T}\|u\|^{2l}_{H_x^1}ds\Bigg),
	\end{aligned}
\end{equation*}
for any $T\in[0,T_0]$. Thus, choosing $T>0$ sufficiently small, we will have 
\begin{equation}\label{priori_esti_Q}
\begin{aligned}
	&\mathbb{E}\sup_{t\in[0,T]}Q^l_t(u,w)\\
	\leq& C\left(\|u_0\|_{H_x^1}, \|w_0\|_{H_x^1},\|k_1\|_{H_x^{1}},\|k_2\|_{H_x^1},T_0,l\right)+\frac{1}{2}\mathbb{E}\sup_{t\in[0,T]}Q^l_t(u,w).
\end{aligned}
\end{equation}
Hence, by dividing $[0,T_0]$ into finite numbers of small intervals, we can finish the proof of \eqref{priori_esti}.
\end{proof}

Next, we deduce the priori estimate of $ (u,v)$  by  limitation.
We first consider the convergence of $ \{(u_{m,n,K},w_{m,n,K})\}_{m,n,K\in\mathbb{N^+}}$ in $ \left[0, T_0\right]$ as $K \uparrow \infty$. 

Let us set 
\begin{equation}\label{uv_m,n}
	\left\{
	\begin{aligned}
		&d u_{m,n}=i\partial_{xx}u_{m,n}dt -iu_{m,n}w_{m,n}dt-i|u_{m,n}|^2u_{m,n}dt\\
		&\qquad\qquad+u_{m,n} P_m\Phi dW_t^{(1)},\\
		&d w_{m,n}=-\partial_{xxx}w_{m,n}dt+  P_n\partial_x(|u_{m,n}|^2)dt-\frac{1}{2}P_n\partial_{x}(w_{m,n}^2)dt\\&\qquad\qquad+ P_m\left(w_{m,n}\Psi dW_t^{(2)}\right),\\
		&u_{m,n}(0)=P_mu_0,\  w_{m,n}(0)=P_mw_0.
	\end{aligned}
	\right.
\end{equation}

\begin{lemma}\label{lem:priori_estimate_of_u,v}
For any $R>0$, $ l\in\mathbb{N^+}$, we have 
\begin{equation}\label{priori_estimate_of_u,v}
\begin{aligned} &\mathbb{E}\|(u_{m,n},w_{m,n})\|^{2l}_{L_{T_0}^\infty \mathcal{H}_x^1}\\
\leq& C(\|u_0\|_{H_x^1}, \|w_0\|_{H_x^1},\|k_1\|_{H_x^{1}},\|k_2\|_{H_x^{1}},T_0,l).
\end{aligned}
\end{equation}
\end{lemma}
\begin{proof}
	 For \eqref{approxiamation_eq}, we define a series of stopping times
	$$
	\iota_K=\inf\left\{t\in\left[0,T_0\right]:|u_{m,n,K}(t)|>K,\ |v_{m,n,K}(t)|>K \right\},
	$$
	for any $ K\in\mathbb{N^+}$.
	
	Because of the well-posedness proved in Proposition \ref{prop:high_globel_well-posedness}, it is clear that for any $K_1\geq K_2, K_1,K_2\in\mathbb{N^+}$, there exists a common set $\mathcal{N}$ satisfying $\mathbb{P}(\mathcal{N})=0$ such that $$u_{m,n,K_1}(t)=u_{m,n,K_2}(t),\ v_{m,n,K_1}(t)=v_{m,n,K_2}(t),\  \forall t\in [ 0,\iota_{K_2}],\ \forall \omega \in\Omega\backslash \mathcal{N}.$$
	
	Then, for any $\omega\in\Omega\backslash \mathcal{N}$, $K\in \mathbb{N^+}$, since $(u_{m,n},v_{m,n})=(u_{m,n,K},v_{m,n,K})$ in $ [0, \iota_K]$, we have
	$$
	\mathbb{E}\| (u_{m,n},w_{m,n})\|^{2l}_{L_{\iota_K}^\infty \mathcal{H}_x^1}\leq C(\|u_0\|_{H_x^1}, \|w_0\|_{H_x^1},\|k_1\|_{H_x^{1}},\|k_2\|_{H_x^{1}},T_0,l).
	$$
	
	Let $\iota^*:=\lim_{K\uparrow \infty} \iota_K\  \mathbb{P}\text{-a.s.}$ Therefore, we have
	$$
    \mathbb{E}\| (u_{m,n},w_{m,n})\|^{2l}_{L_{\iota^*}^\infty \mathcal{H}_x^1}\leq C(\|u_0\|_{H_x^1}, \|w_0\|_{H_x^1},\|k_1\|_{H_x^{1}},\|k_2\|_{H_x^{1}},T_0,l), 
    $$
    which means $\iota^*=T_0,$ $\mathbb{P}\text{-a.s.}$ Hence, we have finished the proof of \eqref{priori_estimate_of_u,v}.
\end{proof}

Let us set $v_{m,n}:=w_{m,n}-U(t)u_0$.

	According to \eqref{priori_estimate_of_u,v}, we have for any $T_0,R>0$, $ l\in\mathbb{N^+}$,
	\begin{equation}\label{priori_estimate_v_m,n}
		\begin{aligned} &\mathbb{E}\|(u_{m,n},v_{m,n})\|^{2l}_{L_{T_0}^\infty \mathcal{H}_x^1}\\
			\leq& C(\|u_0\|_{H_x^1}, \|w_0\|_{H_x^1},\|k_1\|_{H_x^{1}},\|k_2\|_{H_x^{1}},T_0,l).
		\end{aligned}
	\end{equation}

In the next lemma, we will use \eqref{priori_estimate_v_m,n} to get the Bourgain norm priori estimate of $(u_{m,n},v_{m,n})$.

\begin{lemma}\label{lem:priori_estimate_bourgain_norm_u_m,n}
	For any $T_0,R>0$, $l\in\mathbb{N^+}$, we have 
	\begin{equation}\label{m,n_priori_estimate}
		\begin{aligned}
			&\mathbb{E}\|(u_{m,n},v_{m,n})\|^{2l}_{X_{b,1}^{T_0}\times \tilde{Y}_{b,1}^{T_0}}\\
			\leq&C\left(\|u_0\|_{H_x^1}, \|w_0\|_{H_x^1},\|k_1\|_{H_x^{1}},\|k_2\|_{H_x^{1}\cap L_x^1},T_0,l\right).
		\end{aligned}
	\end{equation}
\end{lemma}
\begin{proof}
We introduce the following localized equations and stopping times:
\begin{equation}\label{uv_m,n_localizd}
\left\{
\begin{aligned}
	&u_{m,n}^r(t)=-i\int_{0}^{t} S(t-s)\left(\tilde{u}_{m,n}^r(\tilde{v}_{m,n}^r+U(s)P_mw_0)+|\tilde{u}_{m,n}^r|^2\tilde{u}_{m,n}^r\right)ds\\
	&\qquad \quad+ S(t)P_mu_0+\int_{0}^{t} S(t-s)u_{m,n}^rP_m\Phi dW_s^{(1)},\\
	&v_{m,n}^r(t)=\int_{0}^{t} U(t-r)\partial_xP_n\left(|\tilde{u}_{m,n}^r|^2-\frac{1}{2}(\tilde{v}_{m,n}^r+U(s)P_mw_0)^2\right)ds\\
	&\qquad \quad+\int_{0}^{t}U(t-s) P_m((\tilde{v}_{m,n}^r+U(s)P_mw_0)\Psi )dW_s^{(2)},
\end{aligned}
\right.
\end{equation}
and
$$
\tau_{r,n}^{(1)}=\inf\left\{t\geq 0: \|u_{m,n}^r\|_{X^t_{b,1}}\geq r\right\},\ \tau_{r,n}^{(2)}=\inf\left\{t\geq 0: \|v_{m,n}^r\|_{\tilde{Y}^t_{b,1}}\geq r \right\},
$$
for any $r>0$ and $n\in\mathbb{N^+}$. Here, 
$$ 
\tilde{u}^r_{m,n}(x,t)=\theta_r(\|u^r_{m,n}\|_{X^t_{b,1}})u^{r}_{m,n}(x,t) , \ \tilde{v}^r_{m,n}(x,t)=\theta_r(\|v^r_{m,n}\|_{\tilde{Y}^t_{b,1}})v^r_{m,n}(x,t).
$$
Moreover,  we set
$$
\tau_{r,n}= \tau_{r,n}^{(1)}\wedge \tau_{r,n}^{(2)},\ \mathbb{P}\text{-a.s.}
$$

For the sake of simplicity, since the proof is independent to $n$ in this lemma, we use the notation $ \tau_{r}$ to represent $\tau_{r,n}$.

It can be proved like Theorem \ref{thm:local_well-posedness} that for any $l\in\mathbb{N^+}$, \eqref{uv_m,n_localizd} exists a unique solution in $L^{2l}\left(\Omega;X_{b,1}^{T_0\wedge\tau_{r}}\times \tilde{Y}_{b,1}^{T_0\wedge\tau_{r}}\right)$.

Next, inspired by a pathwise method in  \cite{kdvmuti}, we will estimate 
$$ \mathbb{E} \|(u_{m,n},v_{m,n})\|^{2l}_{X_{b,1}^{T_0\wedge\tau_{r}}\times \tilde{Y}_{b,1}^{T_0\wedge\tau_{r}}}.$$

Step 1: We first have
$$
\begin{aligned}
	&\|u_{m,n}(t)\|_{X_{b,1}^{T\wedge\tau_{r}}}\\
	\leq& CT^{1-(a+b)} \|u_{m,n}\|_{X^{T\wedge\tau_{r}}_{b,1}}(\|v_{m,n}\|_{\tilde{Y}^{T\wedge\tau_{r}}_{b,1}}+C(T)\|w_0\|_{H_x^1})
\\
& +CT^{1-(a+b)}\|u_{m,n}\|^3_{X^{T\wedge\tau_{R,r}}_{b,1}}+ C(T_0)\|u_{m,n}\|_{L^\infty_{T_0}H_x^1}\\
&+\left\|\int_{0}^{t}S(t-s)\left[u_{m,n} P_m\Phi dW_s^{(1)}\right]\right\|_{X_{b,1}^{T_0}}
\end{aligned}
$$
and
$$
\begin{aligned}
	&\|v_{m,n}(t)\|_{\tilde{Y} _{b,1}^{T\wedge\tau_{r}}}\\
	\leq&  CT^{1-(a+b)} \left(\|v_{m,n}\|_{\tilde{Y}^{T\wedge\tau_{r}}_{b,1}}^2
	+ C(T)\|w_0\|^2_{H_x^1}+\|u_{m,n}\|^2_{X^{T\wedge\tau_{r}}_{b,1}}\right)\\
	&+\left\|\int_{0}^{t}U(t-s)\left[P_m((v_{m,n}+U(s)P_mw_0)\Psi) dW_s^{(2)}\right]\right\|_{\tilde{Y}_{b,1}^{T_0}},
\end{aligned}
$$
for any $T\in[0,T_0]$.
We set 
$$
\begin{aligned}
	A_1(\omega):=&C(T_0)\|u_{m,n}\|_{L^\infty_{T_0}H_x^1}+C(T_0)\|w_{m,n}\|_{L^\infty_{T_0}H_x^1}\\
	&+\left\|\int_{0}^{t}S(t-s)\left[u_{m,n} P_m\Phi dW_s^{(1)}\right]\right\|_{X_{b,1}^{T_0}}\\
	&+\left\|\int_{0}^{t}U(t-s)\left[P_m((v_{m,n}+U(s)P_mw_0)\Psi) dW_s^{(2)}\right]\right\|_{\tilde{Y}_{b,1}^{T_0}}.
\end{aligned}
$$

Thus, if we choose $T=T(\omega)$ satisfying $ T^{1-(a+b)}\leq  \frac{1}{20CA_1} \wedge \frac{1}{20CA^2_1}$, we have $\|u_{m,n}\|_{X^{T\wedge\tau_{r}}_{b,1}}\leq 2A_1, \|v_{m,n}\|_{\tilde{Y}^{T\wedge\tau_{r}}_{b,1}}\leq 2A_1$.

We define 
$$
v_{m,n}^k(t):= w_{m,n}(t)-U(t-kT)w_{m,n}(kT),\ \forall t\in\left[kT\wedge\tau_{r},(k+1)T\wedge\tau_{r}\right],
$$
It also holds 
$$
\|u_{m,n}\|_{X_{b,1}^{[kT\wedge\tau_{r},(k+1)T\wedge\tau_{r}]}}\leq 2A_1(\omega),\ \left\|v_{m,n}^k\right\|_{\tilde{Y}_{b,1}^{[kT\wedge\tau_{r},(k+1)T\wedge\tau_{r}]}}\leq 2A_1(\omega),
$$
for any $k\in\mathbb{N^+}$.

Step 2. Next, we consider the estimate of $\|w_{m,n}(t)-U(t)P_mw_0\|_{H_x^1\cap\dot{H}_x^{-3/8}}$, for any $t\in [0,\tau_{r}\wedge T_0]$. We have
$$
\begin{aligned}
	&\|w_{m,n}(t)-U(t)P_mw_0\|_{H_x^1\cap\dot{H}_x^{-3/8}}\\
	\leq&\left\|\int_{0}^{t} U(t-s)P_m\left((v_{m,n}+U(s)P_mw_0)\Psi dW_s^{(2)}\right)\right\| _{H_x^1\cap\dot{H}_x^{-3/8}}\\
	&+\sum_{k=0}^{[t/T(\omega)]}\Bigg(\left\|\int_{kT}^{(k+1)T\wedge t} U((k+1)T\wedge t-s) \partial_{x}\left(|u_{m,n}|^2\right)  ds\right\|_{H_x^1\cap\dot{H}_x^{-3/8}}\\
	&+\bigg\|\int_{kT}^{(k+1)T\wedge t} U((k+1)T\wedge t-s) \partial_{x}P_n\Big(\big(v_{m,n}^k(s)\\
	&\qquad\qquad\qquad+ U(s-kT)w_{m,n}(kT)\big)^2 \Big) ds\bigg\|_{H_x^1\cap\dot{H}_x^{-3/8}}\Bigg).
\end{aligned}
$$

By \eqref{restricted_norm_geq_1/2} and Proposition 2.3 in \cite{kdvmuti}, we have
$$
\begin{aligned}
	&\left\|\int_{kT}^{(k+1)T\wedge t} U((k+1)T\wedge t-s) \partial_{x}\left(|u_{m,n}|^2\right)  ds\right\|_{H_x^1\cap\dot{H}_x^{-3/8}}\\
	\leq&\left\|\int_{kT}^{\cdot}U(\cdot-s) \partial_{x}\left(|u_{m,n}|^2 \right)  ds \right\|_{C\left(\left[kT\wedge\tau_{r},(k+1)T\wedge\tau_{r}\right];H_x^1\cap\dot{H}_x^{-3/8}\right)}\\
	\leq&\left\|\int_{kT}^{\cdot}U(\cdot-s) \partial_{x}\left(|u_{m,n}|^2\right)  ds \right\|_{Y^{\left[kT\wedge\tau_{r},(k+1)T\wedge\tau_{r}\right]}_{\tilde{b},1}}\\
	\leq& C(T_0) \|u_{m,n}\|^2_{X_{b,1}^{\left[kT\wedge\tau_{r},(k+1)T\wedge\tau_{r}\right]}}\leq C(T_0)A_1(\omega)^2
\end{aligned}
$$
and 
$$
\begin{aligned}
	&\left\|\int_{kT}^{(k+1)T\wedge t} U(-s) \partial_{x}\Big(\big(v_{m,n}^k(s)+ U(s-kT)w_{m,n}(kT)\big)^2\Big) ds\right\|_{H_x^1\cap\dot{H}^{-3/8}_x}\\
	\leq& \bigg\|\int_{kT}^{\cdot} U(\cdot-s) \partial_{x}\Big(\big(v_{m,n}^k(s)\\
	&\qquad+ U(s-kT)w_{m,n}(kT)\big)^2 \Big) ds \bigg\|_{C\left(\left[kT\wedge\tau_{r},(k+1)T\wedge\tau_{r}\right];H_x^1\cap\dot{H}^{-3/8}_x\right)}\\
	\leq& C\bigg\|\int_{kT}^{\cdot} U(\cdot-s) \partial_{x}\Big(\big(v_{m,n}^k(s)+ U(s-kT)w_{m,n}(kT)\big)^2\Big) ds\bigg\|_{\tilde{Y}^{\left[kT\wedge\tau_{r},(k+1)T\wedge\tau_{r}\right]}_{\tilde{b},1}}\\
	\leq&C(T_0)\left\| \partial_{x}\left(\left(v_{m,n}^k(s)+ U(s-kT)w_{m,n}(kT)\right)^2\right)\right\|_{\tilde{Y}^{\left[kT\wedge\tau_{r},(k+1)T\wedge\tau_{r}\right]}_{-a,1}}\\
	\leq &C(T_0) \left(\left\|v_{m,n}^k\right\|^2_{\tilde{Y}^{\left[kT\wedge\tau_{r},(k+1)T\wedge\tau_{r}\right]}_{b,1}} + \|w_{m,n}(kT)\|^2_{H_x^1}\right)\\
	 \leq &C(T_0)\Bigg(A_1(\omega)^2+C(T_0) \|w_{m,n}\|^2_{L^\infty_{\tau_{r}\wedge T_0}H_x^1}\Bigg),
\end{aligned}
$$
for any $t\in[0,\tau_{r}\wedge T_0]$ and $\tilde{b}\in(1/2,1)$.

Therefore, we have 
\begin{equation}\label{H_x^{-3/8}_esti}
	\|w_{m,n}(t)-U(t)P_mw_0\|_{H_x^1\cap\dot{H}_x^{-3/8}}\leq A_2(\omega),\ \forall t\in[0,\tau_{r}\wedge T_0],
\end{equation}
where 
$$
\begin{aligned}
	A_2(\omega)=&\left\|\int_{0}^{\cdot} U(\cdot-s)P_m\left((v_{m,n}+U(s)P_mw_0)\Psi dW_s^{(2)}\right)\right\|_{L^\infty_{T_0}H_x^1\cap\dot{H}^{-3/8}_x}\\
	&+C(T_0)T^{-1}\cdot A^2_1(\omega).
\end{aligned}
$$

Step 3. 
Writing $ (u_{m,n},v_{m,n})$ by a finite summation, on the one hand, we have
$$
\begin{aligned}
	\|u_{m,n}\|_{X_{b,1}^{T_0\wedge\tau_{r}}}\leq& \sum_{k=0}^{[T_0/T]}\|u_{m,n}\|_{X_{b,1}^{\left[kT\wedge\tau_{r}, (k+1)T\wedge\tau_{r}\right]}}\\
	\leq & C(T_0)T^{-1}\cdot A_1(\omega).
\end{aligned}
$$

 On the other hand, since $v_{m,n}(t)=v_{m,n}^{(k)}(t)+U(t-kT)w_{m,n}(kT)-U(t)P_mw_0$, for any $ t\in[ kT\wedge\tau_{r}, (k+1)T\wedge\tau_{r}]$, we have
$$
\begin{aligned}
	\|v_{m,n}\|_{\tilde{Y}_{b,1}^{T_0\wedge\tau_{r}}}
	\leq &\sum_{k=0}^{[T_0/T]}\|v_{m,n}\|_{\tilde{Y}_{b,1}^{\left[kT\wedge\tau_{r}, (k+1)T\wedge\tau_{r}\right]}}\\
	\leq &\sum_{k=0}^{[T_0/T]}\left\|v_{m,n}^{(k)}\right\|_{\tilde{Y}_{b,1}^{\left[kT\wedge\tau_{r}, (k+1)T\wedge\tau_{r}\right]}}\\
	&+\|U(t-kT)w_{m,n}(kT)-U(t)P_mw_0\|_{\tilde{Y}_{b,1}^{\left[kT\wedge\tau_{r}, (k+1)T\wedge\tau_{r}\right]}}\\
	\leq& \left(\frac{T_0}{T}+1\right) \left(A_1(\omega)+A_2(\omega)\right).
\end{aligned}
$$

Therefore, according to $T^{-1}\geq C A^{1/(1-(a+b))}_1(\omega)\vee C A^{2/(1-(a+b))}_1(\omega)$, \eqref{SI_1}-\eqref{stochastic_H_x^1H_x^{-3/8}} and the representation of $A_1(\omega), A_2(\omega)$, we have
$$
\begin{aligned}
	&\mathbb{E}\|u_{m,n}\|^{2l}_{X_{b,1}^{T_0\wedge\tau_{r}}}\\
	\leq&C(T_0)\mathbb{E}A_1(\omega)^{2l+2l/(1-(a+b))}+C(T_0)\mathbb{E}A_1(\omega)^{2l+4l/(1-(a+b))}\\
	\leq & C\left(T_0,\|k_1\|_{H_x^{1}},\|k_2\|_{H_x^{1}\cap L_x^1},\|u_0\|_{H_x^1},\|w_0\|_{H_x^1},l\right)
\end{aligned}
$$
and similarly 
$$
\begin{aligned}
	&\mathbb{E}\|v_{m,n}\|^{2l}_{\tilde{Y}_{b,1}^{T_0\wedge\tau_{r}}}
	\leq C\left(T_0,\|k_1\|_{H_x^{1}},\|k_2\|_{H_x^{1}\cap L_x^1},\|u_0\|_{H_x^1},\|w_0\|_{H_x^1},l\right),
\end{aligned}
$$
which finish the proof of \eqref{m,n_priori_estimate} by letting $r\uparrow+\infty$.

\end{proof}
Next, we study the convergence of $\{(u_{m,n},v_{m,n})\}_{m,n\in\mathbb{N^+}}$ as $n\uparrow \infty$.
Let us set
\begin{equation}\label{u_m}
	\left\{
	\begin{aligned}
		&d u_{m}=i\partial_{xx}u_{m}dt -iu_{m}(v_{m}+U(t)P_mw_0)dt-i|u_{m}|^2u_{m}dt+u_{m} P_m\Phi dW_t^{(1)},\\
		&d v_{m}=-\partial_{xxx}v_{m}dt+  \partial_x\left(|u_{m}|^2\right)dt-\frac{1}{2}\partial_{x}\left((v_{m}+U(t)P_mw_0)^2\right)dt\\
		&\qquad+ P_m\left((v_{m}+U(t)P_mw_0)\Psi dW_t^{(2)}\right),\\
		&u_{m}(0)=P_mu_0,\  v_{m}(0)=0.
	\end{aligned}
	\right.
\end{equation}

\begin{lemma}\label{lem:convergence_n}
For any $T_0>0$, $l\in\mathbb{N^+}$,  we have
\begin{equation}\label{priori_esti_H_x^1_u_m_v_m}
	\begin{aligned}
		&\mathbb{E}\|(u_m,v_m)\|^{2l}_{L^\infty\left(0,T_0;\mathcal{H}_x^1\right)}\\
		\leq&C\left(\|u_0\|_{H_x^1}, \|w_0\|_{H_x^1},\|k_1\|_{H_x^{1}},\|k_2\|_{H_x^{1}\cap L_x^1},T_0,l\right)
	\end{aligned}
\end{equation}
and
\begin{equation}\label{priori_esti_bourgain_u_m_v_m}
\begin{aligned}
&\mathbb{E}\|(u_m,v_m)\|^{2l}_{X_{b,1}^{T_0}\times \tilde{Y}_{b,1}^{T_0}}\\
\leq&C\left(\|u_0\|_{H_x^1}, \|w_0\|_{H_x^1},\|k_1\|_{H_x^{1}},\|k_2\|_{H_x^{1}\cap L_x^1},T_0,l\right).
\end{aligned}	
\end{equation}
\end{lemma}
\begin{proof}
For any $r>0$ and $n\in\mathbb{N^+}$, let us define some stopping times
$$
\iota_{r,m}^{(1)}=\inf\left\{t\geq 0: \|u_{m}^r\|_{X^t_{b,1}}\geq r\right\},\ \iota_{r,m}^{(2)}=\inf\left\{t\geq 0: \|v_{m}^r\|_{\tilde{Y}^t_{b,1}}\geq r \right\},
$$
and 
$$
\iota_{r,m}= \iota_{r,m}^{(1)}\wedge \iota_{r,m}^{(2)} 
$$
almost surely. Since  the proof of this lemma is independent to $m$, for the sake of simplicity, we will use the notation $\iota_{r}$ to represent $\iota_{r,m}$. Let us also set 
$$
\tilde{\iota}_{r,n}:= \iota_{r}\wedge\tau_{r,n}^{(1)}\wedge\tau_{r,n}^{(2)},
\ \mathbb{P} \text{-a.s.},
$$
where $\tau_{r,n}^{(1)},\ \tau_{r,n}^{(2)}$ are defined in Lemma \ref{lem:priori_estimate_bourgain_norm_u_m,n}. 

Similarly like Theorem \ref{thm:local_well-posedness}, we can prove the local well-posedness of \eqref{u_m} in $X_{b,1}^{T_0\wedge\iota_{r}}\times \tilde{Y}_{b,1}^{T_0\wedge\iota_{r}}$.
For convergence, let us consider the equations that$ (\tilde{u}_{m,n},\tilde{v}_{m,n}):=(u_m-u_{m,n},v_m-v_{m,n})$ satisfies:
\begin{equation}\label{tilde_u,v_m,n}
	\begin{aligned}
		\tilde{u}_{m,n}=&-i\int_{0}^{t\wedge\tilde{\iota}_{r,n}} S(t\wedge\tilde{\iota}_{r,n}-s)\left(v_m\tilde{u}_{m,n}+u_{m,n}\tilde{v}_{m,n}+\tilde{u}_{m,n}U(s)P_mw_0\right)ds\\
		&-i\int_{0}^{t\wedge\tilde{\iota}_{r,n}}S(t\wedge\tilde{\iota}_{r,n}-s)\left(|u_m|^2\tilde{u}_{m,n}+u_{m,n}u_m\bar{\tilde{u}}_{m,n}+u_{m,n}\tilde{u}_{m,n}\bar{u}_{m,n}\right)ds\\
		&+\int_{0}^{t\wedge\tilde{\iota}_{r,n}} S(t\wedge\tilde{\iota}_{r,n}-s) \tilde{u}_{m,n} P_m\Phi dW_s^{(1)},\\
		\tilde{v}_{m,n}=&\int_{0}^{t\wedge\tilde{\iota}_{r,n}} U(t\wedge\tilde{\iota}_{r,n}-s) \left(\partial_{x}P_n\left(u_m\bar{\tilde{u}}_{m,n}+\tilde{u}_{m,n}\bar{u}_{m,n}\right)+\partial_xP_{\geq n}\left(|u_m|^2\right)\right)ds\\
		&-\frac{1}{2}\int_{0}^{t\wedge\tilde{\iota}_{r,n}}U(t\wedge\tilde{\iota}_{r,n}-s)\left(\partial_{x}P_n(\tilde{v}_{m,n}(v_m+v_{m,n}+2U(s)P_mw_0))\right)ds\\
		&+\int_{0}^{t\wedge\tilde{\iota}_{r,n}} U(t\wedge\tilde{\iota}_{r,n}-s)\partial_xP_{\geq n} \left(|u_m|^2-\frac{1}{2}\left(v_{m}+U(s)P_mw_0\right)^2\right)ds\\
		&+\int_{0}^{t\wedge\tilde{\iota}_{r,n}} U(t\wedge\tilde{\iota}_{r,n}-s)  P_m\left(\tilde{v}_{m,n}\Psi dW_s^{(2)}\right),
	\end{aligned}
\end{equation}
for any $t\in\left[0,T_0\right]$.

Therefore, by Lemma \ref{lem:KS_to_S}, Lemma \ref{lem:trilinear_estimate} and Lemma \ref{lem:SS_to_K}, we have 
$$
\begin{aligned}
	&\|\tilde{u}_{m,n}\|_{X_{b,1}^{T\wedge\tilde{\iota}_{r,n}}}\\
	\leq& CT^{1-(a+b)}\Big(\| v_m\|_{\tilde{Y}_{b,1}^{T\wedge\tilde{\iota}_{r,n}}}\| \tilde{u}_{m,n}\|_{X_{b,1}^{T\wedge\tilde{\iota}_{r,n}}}+\|u_{m,n}\|_{X_{b,1}^{T\wedge\tilde{\iota}_{r,n}}}\|\tilde{v}_{m,n}\|_{\tilde{Y}_{b,1}^{T\wedge\tilde{\iota}_{r,n}}}\\
	&\qquad\qquad\qquad+\| \tilde{u}_{m,n}\|_{X_{b,1}^{T\wedge\tilde{\iota}_{r,n}}}\|w_0\|_{H_x^1}\Big)\\
	&+CT^{1-(a+b)}\bigg(\|u_m\|_{X_{b,1}^{T\wedge\tilde{\iota}_{r,n}}}\|\tilde{u}_{m,n}\|_{X_{b,1}^{T\wedge\tilde{\iota}_{r,n}}}\|u_{m,n}\|_{X_{b,1}^{T\wedge\tilde{\iota}_{r,n}}} \\
	&+\|u_m\|^2_{X_{b,1}^{T\wedge\tilde{\iota}_{r,n}}}\|\tilde{u}_{m,n}\|_{X_{b,1}^{T\wedge\tilde{\iota}_{r,n}}}+\|u_{m,n}\|^2_{X_{b,1}^{T\wedge\tilde{\iota}_{r,n}}}\|\tilde{u}_{m,n}\|_{X_{b,1}^{T\wedge\tilde{\iota}_{r,n}}}\bigg)\\
	&+\left\|\int_{0}^{t} S(t-s) \tilde{u}_{m,n} P_m\Phi dW_s^{(1)}\right\|_{X_{b,1}^{T\wedge\tilde{\iota}_{r,n}}}
\end{aligned}
$$
and
$$
\begin{aligned}
	&\|\tilde{v}_{m,n}\|_{\tilde{Y}_{b,1}^{T\wedge\tilde{\iota}_{r,n}}}\\
	\leq& CT^{1-(a+b)}\left(\|u_m\|_{X_{b,1}^{T\wedge\tilde{\iota}_{r,n}}}\|\tilde{u}_{m,n}\|_{X_{b,1}^{T\wedge\tilde{\iota}_{r,n}}}+\|u_{m,n}\|_{X_{b,1}^{T\wedge\tilde{\iota}_{r,n}}}\|\tilde{u}_{m,n}\|_{X_{b,1}^{T\wedge\tilde{\iota}_{r,n}}}\right)\\
	&+CT^{1-(a+b)}\left(\|\tilde{v}_{m,n}\|_{X_{b,1}^{T\wedge\tilde{\iota}_{r,n}}}\left(\|v_{m,n}\|_{\tilde{Y}_{b,1}^{T\wedge\tilde{\iota}_{r,n}}}+\|v_{m}\|_{\tilde{Y}_{b,1}^{T\wedge\tilde{\iota}_{r,n}}}+\|w_0\|_{H_x^1}\right)\right)\\
	&+CT^{1-(a+b)}\Big(\left\|P_{\geq n}\partial_{x}(|u_m|^2)\right\|_{\tilde{Y}_{-a,1}^{T\wedge\tilde{\iota}_{r,n}}}\\
	&\qquad\qquad+\left\|P_{\geq n}\partial_{x}(\left(v_{m}+U(s)P_mw_0\right)^2)\right\|_{\tilde{Y}_{-a,1}^{T\wedge\tilde{\iota}_{r,n}}}\Big)\\
	&+\left\|\int_{0}^{t} U(t-s)  P_m(\tilde{v}_{m,n}\Psi dW_s^{(2)})\right\|_{\tilde{Y}_{b,1}^{T\wedge\tilde{\iota}_{r,n}}}.
\end{aligned}
$$
Thus, we have
$$
\begin{aligned}
	&\mathbb{E}\|\tilde{u}_{m,n}\|^{2l}_{X_{b,1}^{T\wedge\tilde{\iota}_{r,n}}}\\
	\leq& CT^{2l-2l(a+b)}\left(r^{2l}+\|w_0\|^{2l}_{H_x^1}+r^{4l}\right)\mathbb{E}\left(\|\tilde{u}_{m,n}\|^{2l}_{X_{b,1}^{T\wedge\tilde{\iota}_{r,n}}}+\|\tilde{v}_{m,n}\|^{2l}_{\tilde{Y}_{b,1}^{T\wedge\tilde{\iota}_{r,n}}}\right)\\
	&+C\|k_1\|^{2l}_{H_x^1}T^{bl}\mathbb{E}\|\tilde{u}_{m,n} \|^{2l}_{X_{b,1}^{T\wedge\tilde{\iota}_{r,n}}},\\
	&\mathbb{E}\|\tilde{v}_{m,n}\|^{2l}_{\tilde{Y}_{b,1}^{T\wedge\tilde{\iota}_{r,n}}}\\
	\leq &CT^{2l-2l(a+b)}(r^{2l}+\|w_0\|^{2l}_{H_x^1})\mathbb{E}\left(\|\tilde{u}_{m,n}\|^{2l}_{X_{b,1}^{T\wedge\tilde{\iota}_{r,n}}}+\|\tilde{v}_{m,n}\|^{2l}_{\tilde{Y}_{b,1}^{T\wedge\tilde{\iota}_{r,n}}}\right)\\
	&+CT^{2l-2l(a+b)}\mathbb{E}\left\|P_{\geq n}\partial_{x}(|u_m|^2)\right\|^{2l}_{\tilde{Y}_{-a,1}^{T\wedge\tilde{\iota}_{r,n}}}\\
	&+CT^{2l-2l(a+b)}\mathbb{E}\left\|P_{\geq n}\partial_{x}(\left(v_{m}+U(s)P_mw_0\right))\right\|^{2l}_{\tilde{Y}_{-a,1}^{T\wedge\tilde{\iota}_{r,n}}}\\
	&+C\|k_2\|^{2l}_{H_x^1\cap L_x^1}T^{bl}\mathbb{E}\|\tilde{v}_{m,n} \|^{2l}_{\tilde{Y}_{b,1}^{T\wedge\tilde{\iota}_{r,n}}},
\end{aligned}
$$
for proper $a,b$.

Moreover, we have
$$
\begin{aligned}
	&\|\tilde{u}_{m,n}\|_{C\left(\left[0,T\wedge\tilde{\iota}_{r,n}\right];H_x^1\right)}\\
	\leq&\left\|\int_{0}^{t} S(t-s)(v_m\tilde{u}_{m,n}+u_{m,n}\tilde{v}_{m,n}+\tilde{u}_{m,n}U(s)P_mw_0)ds\right\|_{X_{\tilde{b},1}^{T\wedge\tilde{\iota}_{r,n}}}\\
	&+\left\|\int_{0}^{t}S(t-s)\left(|u_m|^2\tilde{u}_{m,n}+u_{m,n}u_m\bar{\tilde{u}}_{m,n}+u_{m,n}\tilde{u}_{m,n}\bar{u}_{m,n}\right)ds\right\|_{X_{\tilde{b},1}^{T\wedge\tilde{\iota}_{r,n}}}\\
	&+\left\|\int_{0}^{t} S(t-s)\tilde{u}_{m,n} P_m\Phi dW_s^{(1)}\right\|_{C\left(\left[0,T\wedge\tilde{\iota}_{r,n}\right];H_x^1\right)},\\
	&\|\tilde{v}_{m,n}\|_{C([0,T\wedge\tilde{\iota}_{r,n}];H_x^1\cap\dot{H}^{-3/8}_x)}\\
	\leq&\left\|\int_{0}^{t} U(t-s) \left(\partial_{x}P_n(u_m\bar{\tilde{u}}_{m,n}+\tilde{u}_{m,n}\bar{u}_{m,n})+\partial_xP_{\geq n}|u_m|^2\right)ds\right\|_{\tilde{Y}_{\tilde{b},1}^{T\wedge\tilde{\iota}_{r,n}}}\\
	&+C\left\|\int_{0}^{t}U(t-s)\left(\partial_{x}P_n(\tilde{v}_{m,n}(v_m+v_{m,n}+2U(s)P_mw_0))\right)ds\right\|_{\tilde{Y}_{\tilde{b},1}^{T\wedge\tilde{\iota}_{r,n}}}\\
	&+\left\|\int_{0}^{t} U(t-s)\partial_xP_{\geq n} \left(|u_m|^2-\frac{1}{2}\left(v_{m}+U(s)P_mw_0\right)^2\right)ds\right\|_{\tilde{Y}_{\tilde{b},1}^{T\wedge\tilde{\iota}_{r,n}}}\\
	&+\left\|\int_{0}^{t} U(t-s)  P_m\left(\tilde{v}_{m,n}\Psi dW_s^{(2)}\right)\right\|_{C\left(\left[0,T\wedge\tilde{\iota}_{r,n}\right];H_x^1\cap\dot{H}^{-3/8}_x\right)},
\end{aligned}
$$
for any $\tilde{b}\in(1/2,1)$, $a+\tilde{b}<1$ and $ T\in[0,T_0]$.
Thus, we have 
$$
\begin{aligned}
	&\mathbb{E}\|\tilde{u}_{m,n}\|^{2l}_{C([0,T\wedge\tilde{\iota}_{r,n}];H_x^1)}\\
	\leq&CT^{2l-2l(a+\tilde{b})}\left(r^{2l}+\|w_0\|^{2l}+r^{4l}\right)\mathbb{E}\left(\|\tilde{u}_{m,n}\|^{2l}_{X_{b,1}^{T\wedge\tilde{\iota}_{r,n}}}+\|\tilde{v}_{m,n}\|^{2l}_{\tilde{Y}_{b,1}^{T\wedge\tilde{\iota}_{r,n}}}\right)\\
	&+C(l,\|k_1\|_{H_x^{1}})\cdot T^l \mathbb{E}\|\tilde{u}_{m,n}\|^{2l}_{C\left(\left[0,T\wedge\tilde{\iota}_{r,n}\right];H_x^1\right)},\\
	&\mathbb{E}\|\tilde{v}_{m,n}\|^{2l}_{C\left(\left[0,T\wedge\tilde{\iota}_{r,n}\right];H_x^1\cap\dot{H}^{-3/8}_x\right)}\\
	\leq &CT^{2l-2l(a+\tilde{b})}(r^{2l}+\|w_0\|_{H_x^1}^{2l})\mathbb{E}\left(\|\tilde{u}_{m,n}\|^{2l}_{X_{b,1}^{T\wedge\tilde{\iota}_{r,n}}}+\|\tilde{v}_{m,n}\|^{2l}_{\tilde{Y}_{b,1}^{T\wedge\tilde{\iota}_{r,n}}}\right)\\
	&+CT^{2l-2l\left(a+\tilde{b}\right)}\mathbb{E}\left\|P_{\geq n}\partial_{x}\left(|u_m|^2\right)\right\|^{2l}_{X_{-a,1}^{T\wedge\tilde{\iota}_{r,n}}}\\
	&+CT^{2l-2l\left(a+\tilde{b}\right)}\mathbb{E}\left\|P_{\geq n}\partial_{x}(\left(v_{m}+U(s)P_mw_0\right))\right\|^{2l}_{\tilde{Y}_{-a,1}^{T\wedge\tilde{\iota}_{r,n}}}\\
	&+C(l,\|k_2\|_{H_x^{1}\cap L_x^1}) T^l\mathbb{E}\|\tilde{v}_{m,n}\|^{2l}_{C([0,T\wedge\tilde{\iota}_{r,n}];H_x^1)},
\end{aligned}
$$
for any $\tilde{b}\in(1/2,1)$, $a+\tilde{b}<1$ and $ T\in[0,T_0]$.

Thus, by choosing $T=T(r,l,a,b,T_0,\|k_1\|_{H_x^1},\|k_2\|_{H_x^1\cap L_x^1},\|w_0\|_{H_x^1})$ sufficiently small, we have that
\begin{equation}\label{convergence_auxiliary_1}
\begin{aligned}
	&\mathbb{E}\left(\|\tilde{u}_{m,n}\|^{2l}_{X_{b,1}^{T\wedge\tilde{\iota}_{r,n}}}+\|\tilde{v}_{m,n}\|^{2l}_{\tilde{Y}_{b,1}^{T\wedge\tilde{\iota}_{r,n}}}\right)\\
	\leq& C \mathbb{E}\bigg(\left\|P_{\geq n}\partial_{x}(|u_m|^2)\right\|^{2l}_{X_{-a,1}^{T\wedge\tilde{\iota}_{r,n}}}+\left\|P_{\geq n}\partial_{x}(\left(v_{m}+U(s)P_mw_0\right))\right\|^{2l}_{\tilde{Y}_{-a,1}^{T\wedge\tilde{\iota}_{r,n}}}\bigg)
\end{aligned}
\end{equation}
and 
\begin{equation}\label{convergence_auxiliary_2}
\begin{aligned}
	&\mathbb{E}\left(\|\tilde{u}_{m,n}\|^{2l}_{C([0,T\wedge\tilde{\iota}_{r,n}];H_x^1)}+\|\tilde{v}_{m,n}\|^{2l}_{C([0,T\wedge\tilde{\iota}_{r,n}];H_x^1\cap\dot{H}^{-3/8}_x)}\right)\\
	\leq& C \mathbb{E}\bigg(\left\|P_{\geq n}\partial_{x}(|u_m|^2)\right\|^{2l}_{X_{-a,1}^{T\wedge\tilde{\iota}_{r,n}}}+\left\|P_{\geq n}\partial_{x}(\left(v_{m}+U(s)P_mw_0\right))\right\|^{2l}_{\tilde{Y}_{-a,1}^{T\wedge\tilde{\iota}_{r,n}}}\bigg).
\end{aligned}
\end{equation}
Here, according to the dominated convergence theorem, we can see that the right sides of \eqref{convergence_auxiliary_1} and \eqref{convergence_auxiliary_2} will 
decrease to $0$ as $n\uparrow \infty$.

Hence, for any fixed $\varepsilon>0$ and $\forall j\in\mathbb{N^+}$, by the Chebyshev inequality, we have that there exists a sequence of  $\{n_j\}_{j\in\mathbb{N^+}}$ such that 
$$
\mathbb{P}\left( \|\tilde{u}_{m,n_j}\|_{X_{b,1}^{T\wedge\tilde{\iota}_{r,n_j}}}+\|\tilde{v}_{m,n_j}\|_{\tilde{Y}_{b,1}^{T\wedge\tilde{\iota}_{r,n_j}}}>r/3\right)<\varepsilon/2^j.
$$
Let us denote 
$$
\mathcal{N}_j:=\left\{\|\tilde{u}_{m,n_j}\|_{X_{b,1}^{T\wedge\tilde{\iota}_{r,n_j}}}+\|\tilde{v}_{m,n_j}\|_{\tilde{Y}_{b,1}^{T\wedge\tilde{\iota}_{r,n_j}}}>r/3\right\}.
$$
It is clear that $\{\chi_{\mathcal{N}_j}\}_{j\in\mathcal{N}^+}$ converges to $0$ in probability. Then, we choose a subsequence of $\{\chi_{\mathcal{N}_j}\}_{j\in\mathcal{N}^+}$ still denoted by $\{\chi_{\mathcal{N}_j}\}_{j\in\mathcal{N}^+}$, which converges to $0$ almost surely.

Therefore, we have
\begin{equation}\label{convergence}
\begin{aligned}
	&\mathbb{E}\left(\|u_{m}\|^{2l}_{X_{b,1}^{T\wedge\iota_{r/3}}}+\|v_{m}\|^{2l}_{\tilde{Y}_{b,1}^{T\wedge\iota_{r/3}}}\right)\\
	\leq& C \mathbb{E}\bigg(\left(\left\|P_{\geq n_j}\partial_{x}(|u_m|^2)\right\|^{2l}_{X_{-a,1}^{T\wedge\tilde{\iota}_{r,n_j}}}+\left\|P_{\geq n_j}\partial_{x}(\left(v_{m}+U(s)P_mw_0\right))\right\|^{2l}_{\tilde{Y}_{-a,1}^{T\wedge\tilde{\iota}_{r,n_j}}}\right)\\
	&\qquad\cdot\chi_{\mathcal{N}^c_j}(\omega) \bigg)+C\mathbb{E}\left(\left(\|u_{m,n_j}\|^{2l}_{X_{b,1}^{T}}+\|v_{m,n_j}\|^{2l}_{\tilde{Y}_{b,1}^{T}}\right)\cdot \chi_{\mathcal{N}_j}(\omega)\right).
\end{aligned}
\end{equation}
Thus, by the dominated convergence theorem, we have 
$$
\begin{aligned}
	&\mathbb{E}\left(\|u_{m}\|^{2l}_{X_{b,1}^{T\wedge\iota_{r/3}}}+\|v_{m}\|^{2l}_{\tilde{Y}_{b,1}^{T\wedge\iota_{r/3}}}\right)\\
\leq&C\left(\|u_0\|_{H_x^1}, \|w_0\|_{H_x^1},\|k_1\|_{H_x^{1}},\|k_2\|_{H_x^{1}\cap L_x^1},T_0,l\right).
\end{aligned}
$$

Similarly, we have 
$$
\begin{aligned}
	&\mathbb{E}\|(u_m,v_m)\|^{2l}_{L^\infty\left(0,T\wedge\iota_{r/3};H_x^1\times \left(H_x^1\cap\dot{H}_x^{-3/8}\right)\right)}\\
	\leq&C\left(\|u_0\|_{H_x^1}, \|w_0\|_{H_x^1},\|k_1\|_{H_x^{1}},\|k_2\|_{H_x^{1}\cap L_x^1},T_0,l\right).
\end{aligned}
$$

Thus, using above bounded fact and \eqref{convergence_auxiliary_1}-\eqref{convergence}, by the  dominated convergence theorem, we have
\begin{equation}\label{u_m,n_v_m,n_bourgain_convergence_1}
	\lim_{n\uparrow \infty}\mathbb{E}\left(\|\tilde{u}_{m,n}\|^{2l}_{X_{b,1}^{T\wedge\iota_{r/3}}}+\|\tilde{v}_{m,n}\|^{2l}_{\tilde{Y}_{b,1}^{T\wedge\iota_{r/3}}}\right)=0
\end{equation}
and 
\begin{equation}\label{u_m,n_v_m,n_H1_x_convergence_2}
	\lim_{n\uparrow \infty}\mathbb{E}\left(\|\tilde{u}_{m,n}\|^{2l}_{C([0,T\wedge\iota_{r/3}];H_x^1)}+\|\tilde{v}_{m,n}\|^{2l}_{C([0,T\wedge\iota_{r/3}];H_x^1\cap\dot{H}_x^{-3/8})}\right)=0.
\end{equation}

Furthermore, by dividing $ \left[0,T_0\wedge\iota_{r} \right]$ into $\left[0,T_0\wedge\iota_{r}\right],\ \big[T_0\wedge\iota_{r},2T_0\wedge\iota_{r}\big],... $ finite numbers of intervals and taking the $L^{2l}_{\omega}\left(H_x^1\times \left(H_x^1\cap\dot{H}_x^{-3/8}\right)\right) $ convergences of every initial values under consideration, we have
\begin{equation}\label{u_m,n_v_m,n_bourgain_convergence}
	\lim_{n\uparrow \infty}\mathbb{E}\left(\|\tilde{u}_{m,n}\|^{2l}_{X_{b,1}^{T_0\wedge\iota_{r/3}}}+\|\tilde{v}_{m,n}\|^{2l}_{\tilde{Y}_{b,1}^{T_0\wedge\iota_{r/3}}}\right)=0
\end{equation}
and 
\begin{equation}\label{u_m,n_v_m,n_H1_x_convergence}
	\lim_{n\uparrow \infty}\mathbb{E}\left(\|\tilde{u}_{m,n}\|^{2l}_{C([0,T_0\wedge\iota_{r/3}];H_x^1)}+\|\tilde{v}_{m,n}\|^{2l}_{C([0,T_0\wedge\iota_{r/3}];H_x^1\cap\dot{H}_x^{-3/8})}\right)=0.
\end{equation}

Therefore, by \eqref{priori_estimate_of_u,v} and \eqref{m,n_priori_estimate}, we have
$$
\begin{aligned}
	&\mathbb{E}\left(\|u_m\|_{X_{b,1}^{T_0\wedge\iota_{r}}}^{2l}+\|v_m\|_{\tilde{Y}_{b,1}^{T_0\wedge\iota_{r}}}^{2l}\right)\\
	\leq&C\mathbb{E}\left(\|u_{m,n}\|_{X_{b,1}^{T_0\wedge\iota_{r}}}^{2l}+\|v_{m,n}\|_{\tilde{Y}_{b,1}^{T_0\wedge\iota_{r}}}^{2l}+\|\tilde{u}_{m,n}\|_{X_{b,1}^{T_0\wedge\iota_{r}}}^{2l}+\|\tilde{v}_{m,n}\|_{\tilde{Y}_{b,1}^{T_0\wedge\iota_{r}}}^{2l}\right)\\
	\leq& C\left(\|u_0\|_{H_x^1}, \|w_0\|_{H_x^1},\|k_1\|_{H_x^{1}},\|k_2\|_{H_x^{1}\cap L_x^1},T_0,l\right),
\end{aligned}
$$
which implies \eqref{priori_esti_bourgain_u_m_v_m} by the monotone convergence theorem.

Similarly, we can prove \eqref{priori_esti_H_x^1_u_m_v_m}.
\end{proof}

Finally, we study the convergence of $m$.

\begin{lemma}\label{lem:m_convergence}
For any $T_0>0$, $l\in\mathbb{N^+}$, we have 
\begin{equation}\label{priori_esti_bourgain_u_v}
	\begin{aligned}
		&\mathbb{E}\|(u,v)\|^{2l}_{X_{b,1}^{T_0}\times \tilde{Y}_{b,1}^{T_0}}\\
		\leq&C\left(\|u_0\|_{H_x^1}, \|w_0\|_{H_x^1},\|k_1\|_{H_x^{1}},\|k_2\|_{H_x^{1}\cap L_x^1},T_0,l\right)
	\end{aligned}	
\end{equation}
and 
\begin{equation}\label{priori_esti_H_x^1_u_v}
	\begin{aligned}
		&\mathbb{E}\|(u,v)\|^{2l}_{C\left([0,T_0];\mathcal{H}_x^1\right)}\\
		\leq&C\left(\|u_0\|_{H_x^1}, \|w_0\|_{H_x^1},\|k_1\|_{H_x^{1}},\|k_2\|_{H_x^{1}\cap L_x^1},T_0,l\right).
	\end{aligned}
\end{equation}
\end{lemma}
\begin{proof}
Just like the proof of Lemma \ref{lem:convergence_n}, let us set
$$
\sigma_R:=\sigma_R^{(1)}\wedge \sigma_R^{(2)},\ \tilde{\sigma}_{R,m}:=\sigma_R\wedge \iota_{r,m},\  \mathbb{P}\text{-a.s.},
$$
for any $ R>0$ and $m\in\mathbb{N^+}$, where $\iota_{r,m}$ are defined in Lemma \ref{lem:convergence_n}.

For any $t\in[0,T_0]$, we consider the following equation that $(\tilde{u}_m,\tilde{v}_m):=(u-u_m,v-v_m)$ satisfies:
\begin{equation}\label{tilde_u_m_v_m}
	\begin{aligned}
		\tilde{u}_{m}
		=&-i\int_{0}^{t\wedge\tilde{\sigma}_{R,m}} S(t\wedge\tilde{\sigma}_{R,m}-s)(v\tilde{u}_{m}+u_{m}\tilde{v}_{m}+\tilde{u}_mU(s)w_0\\
		&\qquad\qquad+u_mU(s)P_{\geq m}w_0)ds+S(t\wedge\tilde{\sigma}_{R,m})P_{\geq m}u_0\\
		&-i\int_{0}^{t\wedge\tilde{\sigma}_{R,m}}S(t\wedge\tilde{\sigma}_{R,m}-s)\left(|u|^2\tilde{u}_{m}+u_{m}u\bar{\tilde{u}}_{m}+u_{m}\tilde{u}_{m}\bar{u}_{m}\right)ds\\
		&+\int_{0}^{t\wedge\tilde{\sigma}_{R,m}} S(t\wedge\tilde{\sigma}_{R,m}-s) \tilde{u}_{m}\Phi dW_s^{(1)}\\
		&+\int_{0}^{t} S(t\wedge\tilde{\sigma}_{R,m}-s) u_m P_{\geq m}\Phi dW_s^{(1)},\\
		\tilde{v}_{m}=&\int_{0}^{t\wedge\tilde{\sigma}_{R,m}} U(t\wedge\tilde{\sigma}_{R,m}-s) \partial_{x}(u_m\bar{\tilde{u}}_{m}+\tilde{u}_{m}\bar{u})ds\\
		&-\frac{1}{2}\int_{0}^{t\wedge\tilde{\sigma}_{R,m}}U(t\wedge\tilde{\sigma}_{R,m}-s)\partial_{x}((\tilde{v}_{m}+U(s)P_{\geq m}w_0)\\
		&\qquad\qquad\cdot(v+v_m+U(s)(I+P_m)w_0))ds\\
		&+\int_{0}^{t\wedge\tilde{\sigma}_{R,m}} U(t\wedge\tilde{\sigma}_{R,m}-s)P_{\geq m}\left((v_m+U(s)w_0)\Psi dW_s^{(2)}\right)\\
		&+\int_{0}^{t\wedge\tilde{\sigma}_{R,m}} U(t\wedge\tilde{\sigma}_{R,m}-s)  \left(\tilde{v}_{m}\Psi dW_s^{(2)}\right).
	\end{aligned}
\end{equation}

It is clear that once we choose $T=T(R,\|w_0\|_{H_x^1},T_0,l)$ sufficiently small, we have
\begin{equation}\label{convergence_m_aux}
\begin{aligned}
	&\mathbb{E}\left(\|\tilde{u}_m\|^{2l}_{X_{b,1}^{T\wedge\tilde{\sigma}_{R,m}}}+\|\tilde{v}_m\|^{2l}_{\tilde{Y}_{b,1}^{T\wedge\tilde{\sigma}_{R,m}}}\right)\\
	\leq& C\left(T_0,r,l,b,\|k_1\|_{H_x^1},\|k_2\|_{H_x^1\cap L_x^1},\|w_0\|_{H_x^1}\right)	\Big(\|P_{\geq m}u_0\|^{2l}_{H_x^1}+\|P_{\geq m}w_0\|^{2l}_{H_x^1}\\
	&+\|P_{\geq m}k_1\|^{2l}_{H_x^1}+\|P_{\geq m/2}k_2\|^{2l}_{H_x^1}+ \mathbb{E}\|P_{\geq m/2}(v_m+U(s)w_0)\|_{Y_{b,1}^{T\wedge\tilde{\sigma}_{R,m}}}^{2l}\Big).
\end{aligned}
\end{equation}

Similarly, we can prove 
\begin{equation}\label{convergence_m_aux_2}
	\begin{aligned}
		&\mathbb{E}\left(\|\tilde{u}_m\|^{2l}_{L_{T\wedge\tilde{\sigma}_{R,m}}^\infty H_x^1}+\|\tilde{v}_m\|^{2l}_{L_{T\wedge\tilde{\sigma}_{R,m}}^\infty H_x^1\cap \dot{H}_x^{-3/8}}\right)\\
		\leq& C\left(T_0,r,l,b,\|k_1\|_{H_x^1},\|k_2\|_{H_x^1\cap L_x^1},\|w_0\|_{H_x^1}\right)	\Big(\|P_{\geq m}u_0\|^{2l}_{H_x^1}+\|P_{\geq m}w_0\|^{2l}_{H_x^1}\\
		&+\|P_{\geq m}k_1\|^{2l}_{H_x^1}+\|P_{\geq m/2}k_2\|^{2l}_{H_x^1}+ \mathbb{E}\|P_{\geq m/2}(v_m+U(s)w_0)\|_{Y_{b,1}^{T\wedge\tilde{\sigma}_{R,m}}}^{2l}\Big).
	\end{aligned}
\end{equation}

Using the fact that 
$$
\mathbb{E}\left\|\int_{0}^{t} S(t-s) u_m P_{\geq m}\Phi dW_s^{(1)}\right\|^{2l}_{X_{b,1}^{T\wedge\tilde{\sigma}_{R,m}}}\leq C(T)\|P_{\geq m}k_1\|^{2l}_{H_x^1}\mathbb{E}\|u_m\|^{2l}_{X_{b,1}^{T\wedge\tilde{\sigma}_{R,m}}},
$$
and 
$$
\begin{aligned}
	&\mathbb{E}\left\|\int_{0}^{t} U(t-s)P_{\geq m}\left((v_m+U(s)w_0)\Psi dW_s^{(2)}\right)\right\|^{2l}_{\tilde{Y}_{b,1}^{T\wedge\tilde{\sigma}_{R,m}}}\\
	\leq&C(T)\Big(\|k_2\|^{2l}_{H_x^1}\mathbb{E}\|P_{\geq m/2}(v_m+U(s)w_0)\|_{Y_{b,1}^{T\wedge\tilde{\sigma}_{R,m}}}^{2l}\\
	&+\|P_{\geq m/2}k_2\|^{2l}_{H_x^1}\mathbb{E}\|v_m+U(s)w_0\|^{2l}_{Y_{b,1}^{T\wedge\tilde{\sigma}_{R,m}}}\Big),
\end{aligned}
$$
which are gotten from the proof of Proposition 2.5 and Lemma 2.6 in \cite{kdvmuti}, we have the right side of \eqref{convergence_m_aux} and \eqref{convergence_m_aux_2} converges to $0$ as $m\uparrow\infty$.

Furthermore, by dividing $ [0,T_0\wedge\tilde{\sigma}_{R,m} ]$ into $[0,T\wedge\tilde{\sigma}_{R,m}],\ [T\wedge\tilde{\sigma}_{R,m} ,2T\wedge\tilde{\sigma}_{R,m}],... $ finite numbers of small intervals and using a similar argument like Lemma \ref{lem:convergence_n}, we can prove 
\begin{equation*}
	\lim_{m\uparrow \infty}\mathbb{E}\left(\|\tilde{u}_{m}\|^{2l}_{X_{b,1}^{T_0\wedge\sigma_R}}+\|\tilde{v}_{m}\|^{2l}_{\tilde{Y}_{b,1}^{T_0\wedge\sigma_R}}\right)=0
\end{equation*}
and 
\begin{equation*}
	\lim_{m\uparrow \infty}\mathbb{E}\left(\|\tilde{u}_{m}\|^{2l}_{C([0,T_0\wedge\sigma_R];H_x^1)}+\|\tilde{v}_{m}\|^{2l}_{C([0,T_0\wedge\sigma_R];H_x^1\cap\dot{H}_x^{-3/8})}\right)=0.
\end{equation*}

Hence,  by  \eqref{priori_esti_bourgain_u_m_v_m}, \eqref{priori_esti_H_x^1_u_m_v_m} and the monotone convergence theorem, we can finish the proof.

\end{proof}

\section{Appendix}
In this appendix, we propose a counter-example to interpret that $$\|\cdot\|_{L_x^rL_T^q}\leq C(r,q,T)\|\cdot\|_{L_T^\infty H_x^1}$$
will only be true under the condition  $r\geq q$. The proof is constructive.

\begin{lemma}\label{lem:counter-example}
	For any $q,r\in(0,\infty]$, $s\in\mathbb{R}$,
	$$
	\|u\|_{L_x^rL^q_{t\in[0,1]}}\leq C_{q,r,s} \|u\|_{L^\infty_{t\in[0,1]}H_x^s}
	$$
	can only be true under the condition $q\leq r$.
\end{lemma}
\begin{proof}
	Let $\varphi\in C_0^\infty(\mathbb{R})$, $\text{supp} \varphi \subset (-2,2)$ and $\varphi(x)=1,\ \forall x\in[0,1] $. We construct a series of $\{u_n\}$ as follow:
	$$
	u_n(t,x)=\varphi(x-j),\ t\in(j/n,j+1/n), j\in\{0,1,2,...,n-1\}.
	$$
	Thus, we have 
	$$
	C\sim \|u_n\|_{L^\infty_tH_x^1}\geq\|u_n\|_{L_x^rL_t^q}\geq Cn^{1/r-1/q},
	$$
	which implies $q\leq r$.
\end{proof}

%

\phantomsection
\bibliographystyle{amsplain}
\addcontentsline{toc}{section}{References}
\bibliography{reference}

\providecommand{\bysame}{\leavevmode\hbox to3em{\hrulefill}\thinspace}
\providecommand{\MR}{\relax\ifhmode\unskip\space\fi MR }
\providecommand{\MRhref}[2]{%
  \href{http://www.ams.org/mathscinet-getitem?mr=#1}{#2}
}
\providecommand{\href}[2]{#2}
\begin{thebibliography}{10}

\bibitem{77dynamics}
K.~Appert and J.~Vaclavik, \emph{Dynamics of coupled solitons}, Phys. Fluids
  \textbf{20} (1977), no.~11, 1845--1849.

\bibitem{14zhang}
V.~Barbu, M.~R{\"o}ckner, and D.~Zhang, \emph{Stochastic nonlinear
  schr{\"o}dinger equations with linear multiplicative noise: rescaling
  approach}, Journal of Nonlinear Science \textbf{24} (2014), no.~3, 383--409.

\bibitem{benea2016multiple}
C.~Benea and C.~Muscalu, \emph{{Multiple vector-valued inequalities via the
  helicoidal method}}, Anal. \& PDE \textbf{9} (2016), no.~8, 1931--1988.

\bibitem{chenguguo}
J.~Chen, F.~Gu, and B.~Guo, \emph{{On the global well-posedness of the
  stochastic Schr\"{o}dinger-Korteweg-de Vries system}}, Journal of
  Mathematical Analysis and Applications \textbf{550} (2025), no.~1, 129518.

\bibitem{chen}
Jie Chen, \emph{Strichartz type estimates for solutions to the schr\"{o}dinger
  equation}, Proceedings of the American Mathematical Society \textbf{152}
  (2024), 3941--3953.

\bibitem{corcho2007well}
A.~Corcho and F.~Linares, \emph{{~~Well-posedness for the
  Schr{\"o}dinger-Korteweg-de Vries system}}, Tran. Amer. Math. Soc.
  \textbf{359} (2007), no.~9, 4089--4106.

\bibitem{de09}
A.~De~Bouard and A.~Debussche, \emph{{Soliton dynamics for the Korteweg-de
  Vries equation with multiplicative homogeneous noise}}, Electronic Journal of
  Probability \textbf{14}, 1727 -- 1744.

\bibitem{KdVH1}
A.~de~Bouard and A.~Debussche, \emph{{On the Stochastic Korteweg–de Vries
  equation}}, J. Funct. Anal. \textbf{154} (1998), no.~1, 215--251.

\bibitem{de1999stochastic}
\bysame, \emph{{A stochastic nonlinear Schr\"{o}dinger equation with
  multiplicative noise}}, Commun. Math. Phys. \textbf{205} (1999), 161--181.

\bibitem{schrodingerH1}
\bysame, \emph{{The stochastic nonlinear Schr\"{o}dinger equation in $H^1$ }},
  Stoch. Anal. Appl. \textbf{21} (2003), no.~1, 97--126.

\bibitem{kdvmuti}
A.~De~Bouard and A.~Debussche, \emph{The korteweg-de vries equation with
  multiplicative homogeneous noise}, Stochastic Differential Equations: Theory
  And Applications \textbf{A Volume in Honor of Professor Boris L Rozovskii}
  (2007), 113--133.

\bibitem{de1999white}
A.~de~Bouard, A.~Debussche, and Y.~Tsutsumi, \emph{{White noise driven
  Korteweg--de Vries equation}}, J. Funct. Anal. \textbf{169} (1999), no.~2,
  532--558.

\bibitem{erdogan_tzirakis_2016}
M.~Burak Erdo\v{g}an and Nikolaos Tzirakis, \emph{Dispersive partial
  differential equations: Wellposedness and applications}, London Mathematical
  Society Student Texts, Cambridge University Press, 2016.

\bibitem{77theory}
J.~Gibbons, S.~Thornhill, M.~Wardrop, and D.~Ter~Haar, \emph{{On the theory of
  Langmuir solitons}}, J. Plasma Phys. \textbf{17} (1977), no.~2, 153--170.

\bibitem{classical_fourier_analysis}
Loukas Grafakos, \emph{Classical fourier analysis}, vol. Graduate Texts in
  Mathematics 249, Springer, 2014.

\bibitem{Guo}
B.~Guo, \emph{{The global existence and uniqueness of solution for peroidic
  boundary value problem and initial value problemn for a class of coupled
  system of SchrSdinger-KdV equations.}}, Acta Mathematica Sinica \textbf{26}
  (1983), no.~5, 513--532.

\bibitem{H1H1}
B.~Guo and C.~Miao, \emph{{Well-posedness of the Cauchy problem for the coupled
  system of the Schrödinger-KdV equations}}, Acta Math. Sinica \textbf{15}
  (1999), no.~2, 215--224.

\bibitem{guo2010well}
Z.~Guo and Y.~Wang, \emph{{On the well-posedness of the
  Schr{\"o}dinger-Korteweg- de Vries system}}, J. Differential Equations
  \textbf{249} (2010), no.~10, 2500--2520.

\bibitem{Tao2006}
Terence Tao, \emph{Nonlinear dispersive equations : local and global analysis},
  2006.

\end{thebibliography}

\end{document}